\documentclass[12pt, reqno, a4paper]{amsart}

\usepackage{ amssymb, amsmath, enumerate, amsfonts, amsthm, mathrsfs, url, bm}

\makeatletter
\@namedef{subjclassname@2020}{%
	\textup{2020} Mathematics Subject Classification}
\makeatother

\usepackage{xcolor}  	
\usepackage[backref]{hyperref}
\hypersetup{
	colorlinks,
	linkcolor={blue!60!black},
	citecolor={blue!60!black},
	urlcolor={red!60!black}
}

\usepackage{color}

\advance \topmargin by -\headheight
\advance \topmargin by -\headsep
\evensidemargin \oddsidemargin
\numberwithin{equation}{section}

\newtheorem{theorem}{Theorem}[section]
\newtheorem{lemma}[theorem]{Lemma}    
\newtheorem{corollary}[theorem]{Corollary}
\newtheorem{proposition}[theorem]{Proposition}

\theoremstyle{remark}

\theoremstyle{definition}
\newtheorem*{definition}{Definition}
\newtheorem*{remark}{Remark}

\newtheorem*{question*}{Question}

\newcommand{\vc}{\mathbf{c}}
\newcommand{\vd}{\mathbf{d}}

\newcommand{\vn}{\mathbf{n}}

\newcommand{\vu}{\mathbf{u}}
\newcommand{\vv}{\mathbf{v}}

\newcommand{\vx}{\mathbf{x}}
\newcommand{\vy}{\mathbf{y}}

\newcommand{\vA}{\mathbf{A}}
\newcommand{\vB}{\mathbf{B}}
\newcommand{\vC}{\mathbf{C}}

\newcommand{\vJ}{\mathbf{J}}
\newcommand{\vM}{\mathbf{M}}

\newcommand{\valpha}{\bm{\alpha}}
\newcommand{\vbeta}{\bm{\beta}}

\newcommand{\vtheta}{\bm{\theta}}

\newcommand{\vzero}{\mathbf{0}}

\newcommand{\cB}{\mathcal{B}}

\newcommand{\cF}{\mathcal{F}}
\newcommand{\cM}{\mathcal{M}}
\newcommand{\cN}{\mathcal{N}}
\newcommand{\cU}{\mathcal{U}}

\newcommand{\rB}{\mathrm{B}}

\newcommand{\fkF}{\mathfrak{F}}

\newcommand{\m}{\mathfrak{m}}

\newcommand{\N}{\mathbb{N}}
\newcommand{\Z}{\mathbb{Z}}
\newcommand{\Zno}{\mathbb{Z}\setminus\{0\}}
\newcommand{\Q}{\mathbb{Q}}
\newcommand{\R}{\mathbb{R}}
\newcommand{\C}{\mathbb{C}}
\newcommand{\T}{\mathbb{T}}

\newcommand{\supp}{\mathrm{supp}}

\newcommand{\rank}{\mathrm{rank}}

\newcommand{\eps}{\varepsilon}

\newcommand{\str}{\mathrm{str}}
\newcommand{\sml}{\mathrm{sml}}
\newcommand{\unf}{\mathrm{unf}}

\newcommand{\rmq}{\mathrm{q}}

\newcommand{\parQ}{q}

\makeatletter
\let\@@pmod\pmod
\DeclareRobustCommand{\pmod}{\@ifstar\@pmods\@@pmod}
\def\@pmods#1{\mkern4mu({\operator@font mod}\mkern 6mu#1)}
\makeatother

\begin{document}

\title[Partition regularity for systems of diagonal equations]{Partition regularity for systems of diagonal equations}
\author{Jonathan Chapman}
\address{Department of Mathematics\\ University of Manchester\\ Oxford Road\\ Manchester\\ M13 9PL\\ UK}
\email{jonathan.chapman@manchester.ac.uk}


\date{\today}

\begin{abstract}
We consider systems of $n$ diagonal equations in $k$th powers. Our main result shows that if the coefficient matrix of such a system is sufficiently non-singular, then the system is partition regular if and only if it satisfies Rado's columns condition. Furthermore, if the system also admits constant solutions, then we prove that the system has non-trivial solutions over every set of integers of positive upper density.
\end{abstract}

\maketitle

\setcounter{tocdepth}{1}
\tableofcontents

\section{Introduction}\label{secIntro}

A system of equations is said to be \emph{(non-trivially) partition regular} over a set $S$ if, whenever we finitely colour $S$, we can always find a (non-trivial) solution to our system of equations over $S$ such that each variable has the same colour.
Here, we refer to a solution $\vx=(x_{1},\ldots,x_{s})$ as being \emph{non-trivial} if $x_{i}\neq x_{j}$ for all $i\neq j$. When $S$ is the set of positive integers $\N$, we typically omit $S$ and refer to the system of equations as being (non-trivially) partition regular.

A foundational result in the field of arithmetic Ramsey theory is \emph{Rado's criterion} \cite[Satz IV]{rado}, which classifies all (finite) systems of homogeneous linear equations that are partition regular over $\N$.
\begin{definition}[Columns condition]
	Let $\vM$ be an $n\times s$ matrix with rational entries. Let $\vc^{(1)},\ldots,\vc^{(s)}$ denote the columns of $\vM$. We say that $\vM$ obeys the \emph{columns condition} if there exists a partition $\{1,2,\ldots,s\}=J_{1}\cup\cdots\cup J_{k}$ such that $\sum_{j\in J_{1}}\vc^{(j)}=\vzero$,
	and, for each $1<t\leqslant k$,
	\begin{equation*}
	\sum_{j\in J_{t}}\vc^{(j)}\in\langle \vc^{(r)}: r\in J_{1}\cup\cdots\cup J_{t-1}\rangle_{\Q}.
	\end{equation*}
	Here $\langle V\rangle_{\Q}$ denotes the $\Q$-linear span of a set of vectors $V$ with rational entries.\footnote{By convention, $\langle\emptyset\rangle_{\Q}:=\{\vzero \}$.}
\end{definition}

\noindent\textbf{Rado's criterion.} \emph{Let $\vM$ be an $n\times s$ non-empty matrix of rank $n$ (over $\Q$) with integer entries. The system of equations $\vM\vx=\vzero$
is partition regular if and only if $\vM$ obeys the columns condition.}\vspace{2mm}

Recent developments in the study of partition regularity are motivated by the desire to extend Rado's classification to systems of non-linear equations. Chow, Lindqvist, and Prendiville \cite{clp} recently established the following generalisation of Rado's criterion for diagonal polynomial equations in sufficiently many variables.

\begin{theorem}[{\cite[Theorem 1.3]{clp}}]\label{thmCLP}
	For each $k\in\N$, there exists $s_{0}(k)\in\N$ such that the following is true. If $s\geqslant s_{0}(k)$ and $a_{1},\ldots,a_{s}\in\Zno$, then the equation
	\begin{equation}\label{eqnSinglePoly}
	a_{1}x_{1}^{k}+\cdots+a_{s}x_{s}^{k}=0
	\end{equation}
	is non-trivially partition regular over $\N$ if and only if there exists a non-empty set $I\subseteq[s]$ such that $\sum_{i\in I}a_{i}=0$. Moreover, one may take $s_{0}(1)=3$, $s_{0}(2)=5$, $s_{0}(3)=8$, and $s_{0}(k)=k(\log k+\log\log k + 2 + O(\log\log k/\log k))$.
\end{theorem}

The main purpose of this article is to generalise this theorem to suitably non-singular systems of equations. 
To state our result, we require the following notation. Given a vector $\vv=(v_{1},\ldots,v_{s})\in\Q^{s}$, we define the \emph{support} of $\vv$ to be the set $\supp(\vv):=\{j\in\{1,\ldots,s\}: v_{j}\neq 0\}$.

Our main result is the following.

\begin{theorem}[Partition regularity for diagonal polynomial systems]\label{thmMain}
	Let $k,n,s\in\N$, with $k\geqslant 2$, and let $\vM=(a_{i,j})$ be an $n\times s$ matrix with integer entries. Suppose that the following condition holds:
	\begin{enumerate}[\upshape(I)]
		\item\label{cndNonSing} for every non-empty set $\{\vv^{(1)},\ldots,\vv^{(d)} \}\subseteq\Q^{s}$ of linearly independent non-zero vectors in the row space of $\vM$, we have
		\begin{equation*}
		\left\lvert\bigcup_{i=1}^{d}\supp\left(\vv^{(i)}\right)\right\rvert\geqslant dk^{2}+1.
		\end{equation*}
	\end{enumerate}
	Then the system of equations
	\begin{align}
	a_{1,1}x_{1}^{k}+\cdots a_{1,s}x_{s}^{k}&=0;\nonumber\\
	&\vdots\label{eqnCLPsys}\\
	a_{n,1}x_{1}^{k}+\cdots a_{n,s}x_{s}^{k}&=0\nonumber
	\end{align}
 is non-trivially partition regular if and only if $\vM$ obeys the columns condition.
\end{theorem}

The necessity of the columns condition in Theorem \ref{thmMain} was originally established by Lefmann \cite[Theorem 2.1]{lefmann}. In fact, Lefmann observed that (\ref{eqnCLPsys}) is (non-trivially) partition regular over $\N$ if and only if the linear system $\vM\vx=\vzero$ is (non-trivially) partition regular over the set $\{n^{k}:n\in\N \}$.

We remark that whilst it may be possible to weaken condition (\ref{cndNonSing}) of Theorem \ref{thmMain}, it cannot be removed entirely. This is because the columns condition is not sufficient to ensure that (\ref{eqnCLPsys}) has non-zero integral solutions. To see this, consider the matrix
\begin{align*}
\vM=
\begin{pmatrix}
1 & -2 & 1 & 0\\
1 & -1 & 0 & 1
\end{pmatrix}
.
\end{align*}
Since the first three columns of $\vM$ span $\Q^{2}$ and sum to $\vzero$, we see that the columns condition holds. Observe that if $(x,y,z,d)\in\ker\vM$, then $\{x,y,z\}$ is an arithmetic progression with common difference $d$. However, Fermat's right triangle theorem implies that there are no non-trivial arithmetic progression of length $3$ in the squares whose common difference is also a square (see \cite{conradcong} for further details). Thus, the corresponding system of quadric equations
\begin{align*}
x^{2}+z^{2}&=2y^{2};\\
x^{2}+d^{2}&=y^{2}
\end{align*}
has no solutions over $\N$, and is therefore not partition regular. 

In general, it is a very difficult problem in number theory and algebraic geometry to determine whether an arbitrary system of Diophantine equations has any non-trivial solutions. For this reason, we are required to impose some form of non-singularity condition on our system, such as (\ref{cndNonSing}), so that we may use the Hardy-Littlewood circle method to count solutions. This allows us to develop the tools introduced by Chow, Lindqvist, and Prendiville \cite{clp} so that we may establish partition regularity for systems of equations.

In the case when $k=n=2$ and the rows of $\vM$ are linearly independent, condition (\ref{cndNonSing}) is equivalent to the assertion that $s\geqslant 9$ and every non-zero vector in the row space of $\vM$ has at least $5$ non-zero entries. This observation shows that Theorem \ref{thmMain} confirms a conjecture of Chow, Lindqvist, and Prendiville \cite[Conjecture 3.1]{clp} concerning pairs of quadric equations in $9$ variables.

\begin{corollary}[Partition regularity for pairs of quadrics]\label{corPair}
	Let $s\in\N$, and let $\vM=(a_{i,j})$ be a $2\times s$ integer matrix with no zero columns. Suppose that $s\geqslant 9$, and that every non-zero vector in the row space of $\vM$ has at least $5$ non-zero entries.
	Then the system of equations
	\begin{align}
	a_{1,1}x_{1}^{2}+\cdots + a_{1,s}x_{s}^{2}&=0;\nonumber\\
	a_{2,1}x_{1}^{2}+\cdots + a_{2,s}x_{s}^{2}&=0\nonumber
	\end{align}
	is non-trivially partition regular if and only if $\vM$ obeys the columns condition.
\end{corollary}

\subsection{Density regularity}

A set of positive integers $A$ is said to have \emph{positive upper density} if
\begin{equation*}
\limsup_{N\to\infty}\frac{|A\cap\{1,\ldots,N\}|}{N}>0.
\end{equation*}
We call a system of equations \emph{(non-trivially) density regular} (over $\N$) if the system has a (non-trivial) solution over every set of positive integers with positive upper density. It follows from  Szemer\'{e}di's theorem \cite{szem} that the linear homogeneous system $\vA\vx=\vzero$ is density regular if and only if the columns of $\vA$ sum to zero (see \cite[Fact 4]{franklgrahamrodl}).

More recent work on density regularity has focused on non-linear configurations. Browning and Prendiville \cite{brownprend} obtained quantitative bounds for the largest subset of $\{1,\ldots,N\}$ lacking non-trivial solutions to (\ref{eqnSinglePoly}) for $k=2$.
In particular, for $k=2$ and $s\geqslant 5$, they prove that (\ref{eqnSinglePoly}) is non-trivially density regular if and only if $a_{1}+\cdots+a_{s}=0$. For $k\geqslant 3$ and $s\geqslant k^{2}+1$, Chow \cite{chow} showed that (\ref{eqnSinglePoly}) has non-trivial solutions over every relatively dense subset of the primes.

Our second main theorem classifies, in a quantitative manner, density regular systems of diagonal equations satisfying condition (\ref{cndNonSing}) of Theorem \ref{thmMain}.

\begin{theorem}[Density regularity for polynomial systems]\label{thmDens}
	Let $k,n,s\in\N$, with $k\geqslant 2$, and let $\vM=(a_{i,j})$ be an $n\times s$ matrix with integer entries and no zero columns. Let $\delta>0$.
	If $\vM$ satisfies condition {\upshape(\ref{cndNonSing})} of Theorem \ref{thmMain} and the columns of $\vM$ sum to $\vzero$, then there exists a constant $c_{1}=c_{1}(\delta,k,\vM)>0$ and a positive integer $N_{1}=N_{1}(\delta,k,\vM)\in\N$ such that the following is true. If $N\geqslant N_{1}$ and $A\subseteq\{1,2,\ldots,N\}$ satisfies $|A|\geqslant\delta N$, then there are at least $c_{1}N^{s-kn}$ non-trivial solutions $\vx=(x_{1},\ldots,x_{s})\in A^{s}$ to (\ref{eqnCLPsys}).
\end{theorem}

\subsection{Structure of the paper}

We begin in \S\ref{secNotPrelim} by introducing all the general notation and conventions used throughout this article. We also state the restriction estimates that are needed to control the counting operators studied in subsequent sections.

In \S\ref{secQuasi} we investigate the structure and properties of matrices obeying condition (\ref{cndNonSing}) of Theorem \ref{thmMain}. This allows us to establish a relative generalised von Neumann theorem for counting operators which count solutions to (\ref{eqnCLPsys}). This is the key result of this paper, as it is integral to the arguments developed in subsequent sections to prove Theorem \ref{thmMain} and Theorem \ref{thmDens}. Furthermore, we prove that condition (\ref{cndNonSing}) implies that the set of trivial solutions to (\ref{eqnCLPsys}) is sparse in the set of all solutions.

In \S\ref{secIndCol} we introduce the notion of multiplicative syndeticity, and recall the induction on colours argument from \cite{clp}. This allows us to reduce both Theorem \ref{thmMain} and Theorem \ref{thmDens} to a result, Theorem \ref{thmQuadMain}, concerning solutions of (\ref{eqnCLPsys}) over dense and multiplicatively syndetic sets.

In \S\ref{secLinW} we use the linearisation and $W$-trick procedures from \cite{clp} and \cite{lindqvist} to reduce Theorem \ref{thmQuadMain} to a `linearised' version (Theorem \ref{thmLinMain}). We also note that only a special case of Theorem \ref{thmLinMain} is needed to prove Theorem \ref{thmDens}, and that this case follows from existing results (such as {\cite[Theorem 2]{franklgrahamrodl}}).

Finally, in \S\ref{secArl} we use the arithmetic regularity lemma of Green \cite{greenARL} to prove Theorem \ref{thmLinMain}, and thereby prove Theorem \ref{thmMain} and Theorem \ref{thmDens}.

\subsection*{Acknowledgements} The author would like to thank Sean Prendiville for his continual support and encouragement, and for providing informative discussions regarding the use of the arithmetic regularity lemma. We also thank Christopher Frei for his support and helpful comments on an earlier draft of this paper. We are grateful to Trevor Wooley for his comments on Lemma \ref{lemEvenMo}.

\section{Notation and preliminaries}\label{secNotPrelim}

\subsection{Set notation}

The set of positive integers is denoted by $\N$. The set of non-negative integers is denoted by $\Z_{\geqslant 0}$. The set of non-negative real numbers is denoted by $\R_{\geqslant 0}$. Given $X\in\R$, we write $[X]:=\{n\in\N:n\leqslant X\}$. The indicator function of a set $A$ is denoted by $1_{A}$. We follow the convention that, for non-empty sets $A$ and $B$, if $t=0$, then $A\times B^{t}=A$.

We write $\T^{d}:=\R^{d}/\Z^{d}$ to denote the $d$-dimensional torus, which we often identify with $[0,1)^{d}$ in the usual way. We define a metric $(\valpha,\vbeta)\mapsto\lVert\valpha-\vbeta\rVert$ on $\T^{d}$ by
\begin{equation}\label{eqnDefModMet}
\lVert\vtheta\rVert:=\min_{\vn\in\Z^{d}}\left( \sum_{i=1}^{d}|\theta_{i}-n_{i}|\right) =\sum_{i=1}^{d}\left( \min_{n\in\Z}|\theta_{i}-n|\right) .
\end{equation}

\subsection{Asymptotic notation}

Let $f:A\to\C$ and $g:A\to\R_{\geqslant 0}$ be functions defined on some set $A$, and let $\lambda_{1},\ldots,\lambda_{s}$ be some parameters. We write $f\ll_{\lambda_{1},...,\lambda_{s}} g$ if there exists a positive constant $C=C(\lambda_{1},...,\lambda_{s})$ depending only on the parameters $\lambda_{i}$ such that $|f(x)|\leqslant Cg(x)$ for all $x\in A$. We also write $g\gg_{\lambda_{1},\ldots,\lambda_{s}} f$ or $f=O_{\lambda_{1},\ldots,\lambda_{s}}(g)$ to denote this same property.

\subsection{Linear Algebra}
The set of $n\times s$ matrices with entries in a given set $S$ is denoted by $S^{n\times s}$. We allow $n$ or $s$ to be zero, in which case an $n\times s$ matrix is an \emph{empty matrix}. The \emph{row space} of $\vM\in\Q^{n\times s}$ is the $\Q$-linear subspace of $\Q^{s}$ spanned by the rows of $\vM$. Given $k\in\N$ and $\vx=(x_{1},\ldots,x_{s})\in\Q^{s}$, we write $\vx^{\otimes k}:=(x_{1}^{k},\ldots,x_{s}^{k})$.

Given matrices $\vM_{i}\in\Q^{n_{i}\times s_{i}}$ for $1\leqslant i\leqslant r$, a \emph{block upper triangular matrix with diagonal $(\vM_{1},\ldots,\vM_{r})$} is a matrix of the form
\begin{equation*}
	\begin{pmatrix}
		\,\vM_{1} & \vA^{(1,2)} & \ldots & \vA^{(1,r)}\, \\
		\vzero & \vM_{2} & \ldots & \vA^{(2,r)}\\
		\vdots & & \ddots & \vdots\\
		\vzero & \vzero & \ldots & \vM_{r}
	\end{pmatrix}
	,
\end{equation*}
where $\vA^{(i,j)}\in\Q^{n_{i}\times s_{j}}$ for all $1\leqslant i<j\leqslant r$.

Two matrices $\vM,\vM'\in\Q^{n\times s}$ are said to be \emph{equivalent} if one can be obtained from the other by performing column permutations and elementary row operations. Observe that if $\vM'$ can be obtained from $\vM$ by performing elementary row operations, then $\ker(\vM')=\ker(\vM)$. If $\vM'$ can be obtained by applying the permutation $\sigma:[s]\to[s]$ to the columns of $\vM'$, then there is a linear isomorphism between $\ker(\vM)$ and $\ker(\vM')$ given by $(x_{1},\ldots,x_{s})\mapsto(x_{\sigma(1)},\ldots,x_{\sigma(s)})$. We therefore see that properties such as having (non-trivial) solutions over a set $A$ or being partition regular are preserved under coefficient matrix equivalence. Throughout this article we frequently make use of this fact to pass to equivalent matrices when necessary.

\subsection{Norms}

Given a bounded function $f:A\to\C$ on a set $A$, we write $\lVert f\rVert_{\infty}:=\sup_{x\in A}|f(x)|$. Let $d\in\N$ and $p\in\R$ with $p\geqslant 1$. For any function $\psi:\T^{d}\to\C$, we define the $L^{p}$ norm of $\psi$ by
\begin{equation*}
\lVert \psi\rVert_{L^{p}(\T^{d})}:=\left( \int_{\T^{d}}|\psi(\valpha)|^{p}d\valpha\right)^{\frac{1}{p}},
\end{equation*}
whenever the above integral exists and is finite.
For any function $g:\Z^{d}\to\C$ of finite support, we define the $L^{p}$ norm of $g$ by
\begin{equation*}
\lVert g\rVert_{L^{p}(\Z^{d})}:=\left( \sum_{x_{1},\ldots,x_{d}}|g(\vx)|^{p}\right)^{\frac{1}{p}}.
\end{equation*}
If the domain of a complex-valued function $f$ is understood to be either $\T^{d}$ or $\Z^{d}$ for some $d\in\N$, then we write $\lVert f\rVert_{p}$ to denote the $L^{p}$ norm of $f$. We say that $f$ is \emph{integrable} if the $L^{p}$ norm of $f$ exists and is finite for all $p\in[1,\infty]$.

Let $f_{1},\ldots,f_{s}$ be complex-valued integrable functions defined on a set $A$, where $A=\T^{d}$ or $A=\Z^{d}$ for some $d\in\N$. If $A=\Z^{d}$, then assume further that each $f_{i}$ has finite support. Given $p_{1},\ldots,p_{s}\in[1,\infty]$ such that $p_{1}^{-1}+\cdots+p_{s}^{-1}=1$ (with the convention that $\infty^{-1}:=0$), we have \emph{H\"{o}lder's inequality}:
\begin{equation}\label{eqnIneqHold}
\lVert f_{1}\cdots f_{s}\rVert_{1}\leqslant\lVert f_{1}\rVert_{p_{1}}\cdots\lVert f_{s}\rVert_{p_{s}}.
\end{equation}
Taking $s=2$ and $p_{1}=p_{2}=2$, we recover the \emph{Cauchy-Schwarz inequality}:
\begin{equation*}
\lVert f_{1}\cdot f_{2}\rVert_{1}\leqslant\lVert f_{1}\rVert_{2}\lVert f_{2}\rVert_{2}.
\end{equation*}

\subsection{Fourier analysis}\label{subsecFourier}

Given $\alpha\in\T$, we write $e(\alpha):=\exp(2\pi i \alpha)$, where $\alpha$ has been identified with an element of $[0,1)$ in the usual way. Given $\valpha\in\T^{d}$ and $\vx\in\Q^{d}$, we write $\valpha\cdot\vx:=\alpha_{1}x_{1}+\cdots+\alpha_{d}x_{d}$.

Let $f:\Z\to\C$ be a function with finite support. The \emph{(linear) Fourier transform of $f$} is the function $\hat{f}:\T\to\C$ defined by
	\begin{equation*}
	\hat{f}(\alpha):=\sum_{x}f(x)e(\alpha x).
	\end{equation*}
	Let $k\in\N$. The \emph{degree $k$ Fourier transform of $f$} is defined by
	\begin{equation*}
	\fkF_{k}[f](\alpha):=\sum_{x}f(x)e(\alpha x^{k}).
	\end{equation*}
Observe that if $f$ is supported on $[N]$ for some $N\in\N$, then $\fkF_{k}[f]=\hat{F}$, where $F:\Z\to\C$ is defined by
\begin{equation*}
F(x)=
\begin{cases}
f(y),\hspace{1mm}&\text{if }x=y^{k}\text{ for some }y\in[N];\\
0, &\text{otherwise}.
\end{cases}
\end{equation*}

Given $d\in\N$ and $\vx\in\Z^{d}$, we have the \emph{orthogonality relations}:
\begin{equation*}
\int_{\T^{d}}e(\valpha\cdot\vx)d\valpha=
\begin{cases}
1,\hspace{1mm}&\text{if }\vx=\vzero;\\
0, &\text{otherwise}.
\end{cases}
\end{equation*}
For $x\in\Z$ and a function $f:\Z\to\C$ supported on a finite subset of $\Z_{\geqslant 0}$, these relations provide us with the \emph{Fourier inversion formulae}:
\begin{equation*}
f(x)=\int_{\T}\hat{f}(\alpha)e(-\alpha x)d\alpha=\int_{\T}\fkF_{k}[f](\alpha)e(-\alpha x^{k})d\alpha,
\end{equation*}
and \emph{Parseval's identity}: $\lVert f\rVert_{L^{2}(\Z)}=\lVert \hat{f}\rVert_{L^{2}(\T)}$.

Our applications of Fourier analysis frequently make use of \emph{restriction estimates}, which are uniform bounds for exponential sums over arithmetic sets. The precise definition we need is as given in \cite[Definition 5.4]{clp}.

\begin{definition}[Restriction estimate]
	Let $N\in\N$, $K>0$, and $p\geqslant 1$. A function $\nu:\{1,\ldots,N\}\to\R_{\geqslant 0}$ is said to satisfy a \emph{$p$-restriction estimate with constant $K$} if, for every function $f:\{1,\ldots,N\}\to\C$ satisfying $|f|\leqslant\nu$, we have
	\begin{equation*}
	\lVert\hat{f}\rVert_{L^{p}(\T)}^{p}=\int_{\T}|\hat{f}(\alpha)|^{p}d\alpha\leqslant K\lVert\nu\rVert_{1}^{p}N^{-1}.
	\end{equation*}
\end{definition}

As an immediate consequence of Parseval's identity, we have the following restriction estimates for the function $1_{[N]}$.

\begin{lemma}[Linear restriction estimate]\label{lemLinRest}
	Let $N\in\N$, and let $f:[N]\to\C$ be such that $\lVert f\rVert_{\infty}\leqslant 1$. If $p\in\R$ with $p\geqslant 2$, then\footnote{Here we use the convention $\lVert 0\rVert_{\infty}^{0}=1$.}
	\begin{equation*}
		\lVert \hat{f}\rVert_{L^{p}(\T)}^{p}=\int_{\T}|\hat{f}(\alpha)|^{p} d\alpha\leqslant N\lVert \hat{f}\rVert_{\infty}^{p-2}.
	\end{equation*}
\end{lemma}
\begin{proof}
	Parseval's identity gives $\lVert \hat{f}\rVert_{L^{2}(\T)}^{2}=\lVert f\rVert_{L^{2}(\Z)}^{2}\leqslant N$, which implies that
	\begin{equation*}
	\int_{\T}|\hat{f}(\alpha)|^{p} d\alpha\leqslant\lVert \hat{f}\rVert_{\infty}^{p-2}\lVert\hat{f}\rVert_{L^{2}(\T)}^{2}\leqslant N\lVert \hat{f}\rVert_{\infty}^{p-2}.
	\end{equation*}
	
\end{proof}

Observe that this lemma immediately implies that $1_{[N]}$ satisfies a $p$-restriction estimate with constant $1$ for every $p\geqslant 2$. We now seek an analogue of this lemma for the degree $k$ Fourier transform. For this we require the following auxiliary result.

\begin{lemma}\label{lemEvenMo}
	Let $k,N,t\in\N$, and let $\cN(k,t,N)$ denote the number of solutions $(\vx,\vy)\in[N]^{2t}$ to the equation $x_{1}^{k}+\cdots+x_{t}^{k}=y_{1}^{k}+\cdots+y_{t}^{k}$. Let $t_{k}:=\lfloor k^{2}/2 \rfloor$. If $k\geqslant 4$, then
	$\cN(t_{k},k,N)\ll_{k} N^{2t_{k}-k}$.
\end{lemma}

\begin{proof}
	Let $k,N\in\N$ be fixed, with $k\geqslant 4$. For each (measurable) $A\subseteq\T$, let
	\begin{equation*}
	I_{s}(A):=\int_{A}|\fkF_{k}[1_{[N]}](\alpha)|^{s}d\alpha.
	\end{equation*}
	Let $\m=\m_{k}$ be as defined in \cite[\S1]{wooley}. From \cite[Lemma 5.1]{majorVaughan}, it follows that if $s\geqslant k+2$, then $I_{s}(\T\setminus\m)\ll_{s,k}N^{s-k}$. Combining the mean value estimate \cite[Theorem 1.1]{guthDemeterBourgain} with the bound \cite[Theorem 2.1]{wooley}, we deduce that $I_{k(k+1)}(\m)\ll_{\eps,k}N^{k^{2}-1+\eps}$ holds for all $\eps>0$. Following the proof of \cite[Theorem 3.1]{wooley}, we can apply H\"{o}lder's inequality and Hua's lemma to this bound to deduce that $I_{s}(\m)\ll_{s,k} N^{s-k-\delta}$ holds for all $s\geqslant k^{2}-1$, for some $\delta=\delta(k)>0$. The desired result now follows by observing that $\cN(s,k,N)=I_{2s}(\m)+I_{2s}(\T\setminus\m)$ and $2t_{k}\geqslant k^{2}-1$.
\end{proof}

We are now ready to state our analogue of Lemma \ref{lemLinRest} for higher degree Fourier transforms. The following previously appeared in \cite[Lemma D.4]{lindqvist}.

\begin{lemma}[Polynomial restriction estimates]\label{lemQuadRest}
	Let $k\in\N\setminus\{1\}$. Let $N\in\N$, and let $f:[N]\to\C$ be such that $\lVert f\rVert_{\infty}\leqslant 1$. 
	If $p\in\R$ with $p>k^{2}$, then\footnote{Since there are only finitely many $k$ satisfying $k^{2}<p$, the dependence on $k$ of the implicit constant can be removed.}
	\begin{equation*}
	\lVert \fkF_{k}[f]\rVert_{L^{p}(\T)}^{p}=\int_{\T}|\fkF_{k}[f](\alpha)|^{p} d\alpha\ll_{p} N^{p-k}.
	\end{equation*}
\end{lemma} 
\begin{proof}
	The case $k=2$ is due to Bourgain \cite[Eqn. (4.1)]{bourgainQuad}. Now suppose that $k\geqslant 3$. Let $t\in\N$, and let $\cN(k,t,N)$ be as defined in Lemma \ref{lemEvenMo}. Let $p\in\R$ with $p>k^{2}$, and let $f:[N]\to\C$ be such that $\lVert f\rVert_{\infty}\leqslant 1$. If $2t\leqslant p$, then the orthogonality relations imply that
	\begin{equation*}
	\lVert |\fkF_{k}[f]\rVert_{L^{p}(\T)}^{p}\leqslant \lVert\fkF_{k}[f]\rVert_{\infty}^{p-2t}\int_{\T}|\fkF_{k}[f](\alpha)|^{2t} d\alpha\leqslant N^{p-2t}\cN(k,t,N).
	\end{equation*}
	By the work of Vaughan \cite[Theorem 2]{vaughan}, we have $\cN(3,4,N)\ll N^{5}$. This proves the lemma for $k=3$. For $k\geqslant 4$, the result now follows immediately from Lemma \ref{lemEvenMo}.
\end{proof}

\section{Quasi-partitionable matrices}\label{secQuasi}

In this section we investigate the structure and properties of diagonal systems (\ref{eqnCLPsys}) whose coefficient matrices obey condition (\ref{cndNonSing}) of Theorem \ref{thmMain}. We then establish a generalised von Neumann theorem for such systems, and also show that the set of trivial solutions for these systems is sparse in the set of all solutions.

As intimated in the introduction, we need to impose some non-singularity conditions on the coefficient matrix $\vM=(a_{i,j})$ in order to count solutions to the system (\ref{eqnCLPsys}). Stating these conditions requires the following notation.

\begin{definition}[$\mu$ and $\rmq$ functions]
	Let $\vM\in\Q^{n\times s}$ and $0\leqslant d\leqslant n$. We write $\mu(d;\vM)$ to denote the largest number of columns of $\vM$ whose $\Q$-linear span has dimension at most $d$. If $d\leqslant\rank(\vM)$,
	then we define
	\begin{equation*}
		\rmq(d;\vM):=\min\left\lvert \bigcup_{i=1}^{d}\supp\left(\vv^{(i)}\right)\right\rvert,
	\end{equation*}
	where the minimum is taken over all collections of $d$ linearly independent vectors $\vv^{(1)},\ldots,\vv^{(d)}$ in the row space of $\vM$. By convention $\rmq(0;\vM)=|\emptyset|=0$.
\end{definition}

\begin{remark}
	Condition (\ref{cndNonSing}) of Theorem \ref{thmMain} can now be seen to be equivalent to the statement that $\rmq(d;\vM)>dk^{2}$ holds for all $1\leqslant d\leqslant\rank(\vM)$.
\end{remark}

The study of systems of diagonal polynomial equations (\ref{eqnCLPsys})
was initiated in the seminal work of Davenport and Lewis \cite{davlewis}. They established that such systems possess non-zero integer solutions provided they admit non-singular real solutions and that $\rmq(d;(a_{i,j}))$ is sufficiently large in terms of $d$, $k$, and $n$ for all $0\leqslant d\leqslant n$.

In their work on simultaneous diagonal congruences, Low, Pitman, and Wolff \cite{lowpitwolf} discovered that these non-singularity conditions could be better understood by being put in the context of matroid theory. Most notably, they observed that a suitable alternative non-singularity condition could be formulated using a theorem known as \emph{Aigner's criterion}\footnote{Originally proved by Edmonds \cite{edmonds}.} \cite[Proposition 6.45]{aigner}.

\begin{definition}[Partitionable matrix]
	Let $\vM\in\Q^{n\times s}$, for some $n,s\in\N$. Let $k\in\N$. We say that $\vM$ is \emph{$k$-partitionable} if $s=kn$ and the columns of $\vM$ can be partitioned into $k$ disjoint blocks of size $n$ such that each block forms an $n\times n$ non-singular submatrix.
\end{definition}

The following is a special case of Aigner's criterion when the matroid under consideration is a vector matroid.

\begin{lemma}[Aigner's criterion]\label{lemAigCrit}
	Let $\vM\in\Q^{n\times kn}$, for some $k,n\in\N$. Then $\vM$ is $k$-partitionable if and only if $\mu(d;\vM)\leqslant dk$ holds for all $0\leqslant d\leqslant n$.
\end{lemma}
\begin{proof}
	See \cite[Proposition 6.47]{aigner} or \cite[Lemma 1]{lowpitwolf}.
\end{proof}

Br\"{u}dern and Cook \cite[Theorem 1]{brucook} subsequently combined Aigner's criterion with the circle method to count the number of solutions to (\ref{eqnCLPsys}). They consider systems (\ref{eqnCLPsys}) whose coefficient matrices $\vM=(a_{i,j})\in\Z^{n\times s}$ are such that there exists an $n\times (ns_{1}+1)$ submatrix $\vM'$ of $\vM$ such that $\mu(d;\vM')\leqslant ds_{1}$ holds for all $0\leqslant d<n$. Provided $s_{1}=s_{1}(k)$ is sufficiently large, and that (\ref{eqnCLPsys}) has non-singular real and $p$-adic solutions for every prime $p$, they prove that there are $\gg_{\vM} N^{s-kn}$ solutions $\vx\in[N]^{s}$ to (\ref{eqnCLPsys}).

For $n=2$ and $s=9$, such a result had previously been obtained by Cook \cite{cookquad}. In this case, as in Corollary \ref{corPair}, the condition on $\vM$ may be replaced with the condition that every non-zero vector in the row space of $\vM$ has at least $5$ non-zero entries.

From Aigner's criterion, we see that this condition on $\vM$ implies that $\vM$ contains an $n\times (ns_{1}+1)$ submatrix $\vM'$ such that every $n\times ns_{1}$ submatrix of $\vM'$ is partitionable. This leads us to consider what we have termed \emph{quasi-partitionable matrices}.

\begin{definition}[Quasi-partitionable matrix]
	A non-empty matrix $\vM\in\Q^{n\times s}$ is called \emph{quasi-$\parQ$-partitionable} if $s\geqslant n\parQ$ and $\mu(d;\vM)\leqslant d\parQ$ holds for all $0\leqslant d<n$. We say that $\vM$ is \emph{quasi-partitionable} is $\vM\in\Q^{n\times s}$ is quasi-$\parQ$-partitionable for some $\parQ\in\N$.
\end{definition}

\begin{corollary}\label{corParty}
	Let $q\in\N$, and $\vM\in\Q^{n\times s}$. If $\vM$ is quasi-$\parQ$-partitionable, then $\vM$ has full rank and every $n\times n\parQ$ submatrix of $\vM$ is $\parQ$-partitionable.
\end{corollary}
\begin{proof}
	If $\vM$ is quasi-$\parQ$-partitionable, then $s>(n-1)\parQ\geqslant\mu(n-1;\vM)$, which implies that $\vM$ has rank $n$. The rest of the corollary follows immediately from Aigner's criterion.
\end{proof}

We now show that all of the systems (\ref{eqnCLPsys}) we consider, namely those obeying condition (\ref{cndNonSing}) of Theorem \ref{thmMain}, may be `semi-decomposed' into systems whose coefficient matrices are quasi-partitionable.

\begin{lemma}[Decomposition lemma]\label{lemPermStr}
	Let $n,\parQ,s\in\N$, and let $\vM\in\Q^{n\times s}$ be a matrix of rank $n$ with no zero columns. If $\rmq(d;\vM)>d\parQ$ for all $1\leqslant d\leqslant n$, then $\vM$ is equivalent to a block upper triangular matrix with diagonal $(\vM_{1},\ldots,\vM_{r})$, where each $\vM_{i}\in\Z^{n_{i}\times s_{i}}$ is quasi-$\parQ$-partitionable and $s_{i}>n_{i}\parQ$. 
\end{lemma}
\begin{proof}
	We proceed by induction on $n$. If $n=1$, then the hypotheses on $\vM$ imply that $\vM$ is quasi-$\parQ$-partitionable. Suppose then that $n\geqslant 2$, and assume the induction hypothesis that if $1\leqslant n'<n$ and $\vM'\in\Q^{n'\times s'}$ is a full rank matrix with no zero columns which satisfies $\rmq(d;\vM)>d\parQ$ for all $1\leqslant d\leqslant n'$, then $\vM'$ satisfies the conclusion of the theorem.
	
	We may henceforth assume that $\vM$ is not quasi-$\parQ$-partitionable, since otherwise we are done. From the bound $s=\rmq(n;\vM)>n\parQ$, we deduce that there exists a minimal $d_{0}\in[n-1]$ such that $s_{0}:=\mu(d_{0};\vM)>d_{0}\parQ$. Hence, $\vM$ is equivalent to a block upper triangular matrix with diagonal $(\vM_{0},\vM')$, for some full rank matrices $\vM_{0}\in\Q^{d_{0}\times s_{0}}$ and $\vM'\in\Q^{n'\times s'}$ without zero columns.
	
	Note that $\mu(d;\vM_{0})\leqslant d\parQ$ holds for all $0\leqslant d<d_{0}$ by the minimality of $d_{0}$. Since $s_{0}>d_{0}\parQ$, we therefore deduce that $\vM_{0}$ is quasi-$\parQ$-partitionable. We also note that $\rmq(d;\vM')\geqslant\rmq(d;\vM)>d\parQ$ holds for all $0\leqslant d\leqslant n'$. Thus, by the induction hypothesis, $\vM'$ satisfies the conclusion of the lemma, and therefore so does $\vM$.
\end{proof}

\subsection{Fourier control} 

We now use the properties of quasi-partitionable matrices to control counting operators for diagonal polynomial systems (\ref{eqnCLPsys}). More precisely, given a matrix $\vM\in\Z^{n\times s}$ satisfying condition (\ref{cndNonSing}) of Theorem \ref{thmMain}, we consider sums of the form
\begin{equation}\label{eqnOpCount}
\Lambda_{\vM}(f_{1},\ldots,f_{s})=\sum_{\vM\vx=\vzero}f_{1}(x_{1})\cdots f_{s}(x_{s}),
\end{equation}
where $f_{1},\ldots,f_{s}:[N]\to\C$. Such counting operators are frequently treated in the arithmetic combinatorics literature, see \cite{chapman,clp,greensand,gtARL,taoHOFA}. The primary method of understanding counting operators is via a \emph{generalised von Neumann theorem}. This term is used to describe any result which asserts that $\Lambda_{\vM}(f_{1},\ldots,f_{s})\approx\Lambda_{\vM}(g_{1},\ldots,g_{s})$
holds whenever $\lVert f_{i}-g_{i}\rVert$ is `small' for all $i$ with respect to some (semi-)norm $\lVert\cdot\rVert$. Determining precisely which (semi-)norms are admissible for a given family of counting operators is also a widely studied problem in its own right, see \cite{wolfgow} for further details.

In this subsection, we establish a generalised von Neumann theorem using the norm given by $f\mapsto\lVert\hat{f}\rVert_{\infty}$. We begin with the observation that, by orthogonality, the sum (\ref{eqnOpCount}) is equal to the integral
\begin{equation*}
\int_{\T^{n}}\prod_{j=1}^{s}\hat{f}_{j}(\valpha\cdot \vc^{(j)})d\valpha.
\end{equation*}
Here, we have written $\vc^{(j)}\in\Z^{n}$ to denote the $j$th column of $\vM$.
Our goal is to use the $L^{p}$ norms of the $\hat{f}_{j}$ to bound this integral.

Let $\vM_{i}\in\Q^{n_{i}\times s_{i}}$ for $1\leqslant i\leqslant r$, and consider a block upper triangular matrix $\vM$ with diagonal $(\vM_{1},\ldots,\vM_{r})$. Typically, the $j$th column of $\vM$ is denoted by $\vc^{(j)}$. However, as we are interested in utilising the properties of the $\vM_{i}$, it is convenient for us to reindex the columns of $\vM$ using $(i,j)$ where $1\leqslant i\leqslant r$ and $j\in[s_{i}]$. Hence, we write $\vc^{(i,j)}$ to denote the column of $\vM$ which intersects the $j$th column of $\vM_{i}$. Explicitly, $\vc^{(i,j)}$ is the $(j+\sum_{1\leqslant t< i}s_{t})$th column of $\vM$.

\begin{theorem}[$L^{p}$ control for integral operators]\label{thmQuasiCtrl}
	Let $n,\parQ,r,s\in\N$, and let $\vM\in\Z^{n\times s}$ be a block upper triangular matrix with diagonal $(\vM_{1},\ldots,\vM_{r})$, for some non-empty $\vM_{i}\in\Z^{n_{i}\times s_{i}}$. As described above, denote the columns of $\vM$ by $\vc^{(i,j)}\in\Z^{n}$ for $i\in[r]$ and $j\in[s_{i}]$. Let $p_{i}:=s_{i}/n_{i}$ for each $i\in[r]$. If every $\vM_{i}$ is quasi-$\parQ$-partitionable, then for any collection of integrable functions $\psi_{i,j}:\T\to\C$ for $i\in[r]$ and $j\in[s_{i}]$, we have
	\begin{equation}\label{eqnLpCtrl}
		\int_{\T^{n}}\prod_{i=1}^{r}\prod_{j=1}^{s_{i}}|\psi_{i,j}(\valpha\cdot\vc^{(i,j)})|d\valpha\leqslant\prod_{i=1}^{r}\prod_{j=1}^{s_{i}}\lVert\psi_{i,j}\rVert_{L^{p_{i}}(\T)}.
	\end{equation}
\end{theorem}
\begin{proof}
	Let $\cB:=\{\vJ=(J_{1},\ldots,J_{r}):J_{i}\subseteq[s_{i}], |J_{i}|=n_{i}\parQ \}$. Note that
	\begin{equation*}
		|\cB|=\prod_{i=1}^{r}\binom{s_{i}}{n_{i}\parQ}.
	\end{equation*}
	For each $i\in[r]$ and $j\in[s_{i}]$, let $a_{i,j}$ be an arbitrary non-negative real number, and let
	\begin{equation*}
		m_{i}:=\binom{s_{i}-1}{n_{i}\parQ-1}\prod_{\substack{t=1\\ t\neq i}}^{r}\binom{s_{t}}{n_{t}\parQ}=\frac{\parQ|\cB|}{p_{i}}.
	\end{equation*}
	Observe that, for any $j\in[s_{i}]$, the number of $\vJ=(J_{1},\ldots,J_{r})\in\cB$ such that $j\in J_{i}$ is given by $m_{i}$. This provides us with the identity
	\begin{equation}\label{eqnProdId1}
		\prod_{i=1}^{r}\prod_{j=1}^{s_{i}}a_{i,j}=\prod_{\vJ\in\cB}\prod_{i=1}^{r}\prod_{j\in J_{i}}a_{i,j}^{1/m_{i}}.
	\end{equation}
	
	Let $i\in[r]$. By Corollary \ref{corParty}, any set $J\subseteq[s_{i}]$ with $|J|=n_{i}\parQ$ admits a partition $J=I_{1}^{(i,J)}\cup \cdots\cup I_{\parQ}^{(i,J)}$ such that, for each $1\leqslant t\leqslant \parQ$, the columns of $\vM_{i}$ indexed by $I_{t}^{(i,J)}$ form a non-singular $n_{i}\times n_{i}$ matrix. Thus, given any $\vJ\in\cB$, we obtain the identity
	\begin{equation}\label{eqnProdId2}
	\prod_{\vu\in[\parQ]^{r}}\prod_{j\in I_{u_{i}}^{(i,J_{i})}}a_{i,j}=\prod_{t=1}^{\parQ}\prod_{\substack{\vu\in[\parQ]^{r}\\ u_{i}=t}}\prod_{j\in I_{t}^{(i,J_{i})}}a_{i,j}=\prod_{j\in J_{i}}a_{i,j}^{\parQ^{r-1}}.
	\end{equation}
	Combining this identity with (\ref{eqnProdId1}) and applying H\"{o}lder's inequality shows that the left-hand side of (\ref{eqnLpCtrl}) is bounded above by
	\begin{equation*}
		\prod_{\vJ\in\cB}\prod_{\vu\in[\parQ]^{r}}\left(\int_{\T^{n}}\prod_{i=1}^{r}\prod_{j\in I_{u_{i}}^{(i,J_{i})}} |\psi_{i,j}(\valpha\cdot\vc^{(i,j)})|^{p_{i}}d\valpha\right)^{(\parQ^{r}|\cB|)^{-1}}.
	\end{equation*}
	
	Now fix some $\vu\in[\parQ]^{r}$ and $\vJ\in\cB$. For each $i\in[r]$, let $\vA_{i}$ be the $n_{i}\times n_{i}$ submatrix of $\vM_{i}$ formed by the columns of $\vM_{i}$ indexed by $I_{u_{i}}^{(i,J_{i})}$. Our choice of the $I_{u_{i}}^{(i,J_{i})}$ ensures that each $\vA_{i}$ is non-singular. Hence, every $n\times n$ block upper triangular matrix with diagonal $(\vA_{1},\ldots,\vA_{r})$ is non-singular (this can be seen directly from the structure, or by noting that the determinant of such a matrix is given by the product of the determinants of the $\vA_{i}$). Thus, the set of vectors $\{\vc^{(i,j)}:i\in[r], j\in I_{u_{i}}^{(i,J_{i})} \}$ is linearly independent. We may therefore perform a change of variables to deduce that
	\begin{equation*}
		\left( \int_{\T^{n}}\prod_{i=1}^{r}\prod_{j\in I_{u_{i}}^{(i,J_{i})}} |\psi_{i,j}(\valpha\cdot\vc^{(i,j)})|^{p_{i}}d\valpha\right)^{(\parQ^{r}|\cB|)^{-1}}=\prod_{i=1}^{r}\prod_{j\in I_{u_{i}}^{(i,J_{i})}}\lVert\psi_{i,j}\rVert_{L^{p_{i}}(\T)}^{(\parQ^{r-1}m_{i})^{-1}}.
	\end{equation*}
	The theorem now follows from (\ref{eqnProdId1}) and (\ref{eqnProdId2}).
\end{proof}

We now return to the problem of bounding the counting operators (\ref{eqnOpCount}). For the applications in \S\ref{secLinW}, we require control for counting operators with weights $f_{i}$ which may be unbounded as $N\to\infty$. Such a result is given for single equations (\ref{eqnSinglePoly}) in \cite[Lemma C.2]{clp}, and we now provide a generalisation for systems (\ref{eqnCLPsys}).

\begin{lemma}[Relative Fourier control]\label{lemRelCtrl}
	Let $k,n,s\in\N$ with $k\geqslant 2$. Let $p:=k^{2}+\tfrac{1}{2n}$, and let $\eta:=(2k^{2}n+2)^{-1}$. Let $N\in\N$, and suppose that $\nu_{1},\ldots,\nu_{s}:[N]\to\R_{\geqslant 0}$ are non-zero functions which each satisfy a $p$-restriction estimate with constant $K$. Let $\vM\in\Z^{n\times s}$ be a matrix of rank $n$ with no zero columns. If $\vM$ satisfies condition {\upshape(\ref{cndNonSing})} of Theorem \ref{thmMain}, then, for any functions $f_{1},\ldots,f_{s}:[N]\to\C$ such that $|f_{i}|\leqslant\nu_{i}$, we have
	\begin{equation*}
		\left\lvert\sum_{\vM\vx=\vzero}\prod_{i=1}^{s}\frac{f_{i}(x_{i})}{\lVert\nu_{i}\rVert_{1}}\right\rvert\leqslant K^{n}N^{-n} \prod_{i=1}^{s}\left(\frac{\lVert \hat{f}_{i}\rVert_{\infty}}{\lVert\nu_{i}\rVert_{1}}\right)^{\eta}.
	\end{equation*}
\end{lemma}
\begin{proof}
	By applying Lemma \ref{lemPermStr}, and relabelling the $f_{i}$ if necessary, we may assume that $\vM$ is a block upper triangular matrix with diagonal $(\vM_{1},\ldots,\vM_{r})$, where each $\vM_{i}\in\Z^{n_{i}\times s_{i}}$ is quasi-partitionable and satisfies $s_{i}>k^{2}n_{i}$. For each $i\in[r]$ and $j\in[s_{i}]$, let $\psi_{i,j}:\T\to\C$ be defined by $\psi_{i,j}:=\lVert\nu_{l}\rVert_{1}^{-1}\hat{f}_{l}$, where $l\in[s]$ is such that the $l$th column of $\vM$ intersects the $j$th column of $\vM_{i}$.
	
	Let $p_{i}=s_{i}/n_{i}$ for each $i\in[r]$, and note that $p_{i}>p$.
	From orthogonality and Theorem \ref{thmQuasiCtrl} we obtain the bound
	\begin{equation}\label{eqnInitialPsi}
	 \left\lvert\sum_{\vM\vx=\vzero}\prod_{i=1}^{s}\frac{f_{i}(x_{i})}{\lVert\nu_{i}\rVert_{1}}\right\rvert\leqslant\prod_{i=1}^{r}\prod_{j=1}^{s_{i}}\lVert\psi_{i,j}\rVert_{p_{i}}\leqslant\prod_{i=1}^{r}\prod_{j=1}^{s_{i}}\lVert\psi_{i,j}\rVert_{p}^{\frac{p}{p_{i}}}\lVert\psi_{i,j}\rVert_{\infty}^{1-\frac{p}{p_{i}}}.
	\end{equation}
	From the restriction estimates for the $\nu_{i}$, we deduce that
	\begin{equation}\label{eqnKNineq}
	\prod_{i=1}^{r}\prod_{j=1}^{s_{i}}\lVert\psi_{i,j}\rVert_{p}^{\frac{p}{p_{i}}}\leqslant\prod_{i=1}^{r}\prod_{j=1}^{s_{i}}\left( \frac{K}{N}\right)^{\frac{1}{p_{i}}}=\left(\frac{K}{N}\right)^{n}.
	\end{equation}
	Observe that $\eta p_{i}\leqslant p_{i}-p$. Hence, for each $i\in[r]$ and $j\in[s_{i}]$, the hypothesis $\lvert f_{i}\rvert\leqslant\nu_{i}$ implies that $
	\lVert\psi_{i,j}\rVert_{\infty}^{1-\frac{p}{p_{i}}}\leqslant \lVert\psi_{i,j}\rVert_{\infty}^{\eta}\leqslant 1$.
	The lemma may now be deduced from (\ref{eqnInitialPsi}) by invoking (\ref{eqnKNineq}).
\end{proof}

Finally, by applying a telescoping identity, we obtain the desired generalised von Neumann theorem.

\begin{theorem}[Relative generalised von Neumann]\label{thmRelGenNeu}
Let $k,n,s\in\N$ with $k\geqslant 2$. Let $p:=k^{2}+\tfrac{1}{2n}$, and let $\eta:=(2k^{2}n+2)^{-1}$. Let $N\in\N$, and suppose that $\nu_{1},\ldots,\nu_{s},\mu_{1},\ldots,\mu_{s}:[N]\to\R_{\geqslant 0}$ are non-zero functions which each satisfy a $p$-restriction estimate with constant $K$. Let $\vM\in\Z^{n\times s}$ be a matrix of rank $n$ with no zero columns. If $\vM$ satisfies condition {\upshape(\ref{cndNonSing})} of Theorem \ref{thmMain}, then, for any functions $f_{1},\ldots,f_{s},g_{1},\ldots,g_{s}:[N]\to\C$ such that $|f_{i}|\leqslant\nu_{i}$ and $|g_{i}|\leqslant\mu_{i}$, we have
	\begin{align*}
		\left\lvert\sum_{\vM\vx=\vzero}\left( \frac{f_{1}(x_{1})}{\lVert\nu_{1}\rVert_{1}}\cdots\frac{f_{s}(x_{s})}{\lVert\nu_{s}\rVert_{1}}- \frac{g_{1}(x_{1})}{\lVert\mu_{1}\rVert_{1}}\cdots\frac{g_{s}(x_{s})}{\lVert\mu_{s}\rVert_{1}}\right)\right\rvert\\
		\leqslant 2sK^{n}N^{-n} \max_{1\leqslant i\leqslant s}\left\lVert\frac{ \hat{f}_{i}}{\lVert\nu_{i}\rVert_{1}}-\frac{ \hat{g}_{i}}{\lVert\mu_{i}\rVert_{1}}\right\rVert_{\infty}^{\eta}.
	\end{align*}
\end{theorem}

\begin{proof}
	For each $i\in[s]$, define functions $h_{i}:[N]\to\C$ and $\tau_{i}:[N]\to\R_{\geqslant 0}$ by
	\begin{equation*}
		h_{i}:=\frac{f_{i}}{\lVert \nu_{i}\rVert_{1}}-\frac{g_{i}}{\lVert \mu_{i}\rVert_{1}}\quad;\quad\tau_{i}:=\frac{\nu_{i}}{\lVert \nu_{i}\rVert_{1}}+\frac{\mu_{i}}{\lVert \mu_{i}\rVert_{1}}.
	\end{equation*}
	By \cite[Lemma C.1]{clp}, the functions $\tau_{1},\ldots,\tau_{s}$ each satisfy a $p$-restriction estimate with constant $K$. Note that $|h_{i}|\leqslant\tau_{i}$ and $\lVert\tau_{i}\rVert_{1}=2$ for all $i\in[s]$. The result now follows by applying the triangle inequality and  Lemma \ref{lemRelCtrl} to the telescoping identity
	\begin{align*}
		\sum_{\vM\vx=\vzero}\left( \frac{f_{1}(x_{1})}{\lVert \nu_{1}\rVert_{1}}\cdots\frac{f_{s}(x_{s})}{\lVert \nu_{s}\rVert_{1}}- \frac{g_{1}(x_{1})}{\lVert \mu_{1}\rVert_{1}}\cdots\frac{g_{s}(x_{s})}{\lVert \mu_{s}\rVert_{1}}\right)\\
		=2\sum_{r=1}^{s}\sum_{\vM\vx=\vzero}\left(\prod_{i=1}^{r-1}\frac{f_{i}(x_{i})}{\lVert \nu_{i}\rVert_{1}} \right)\frac{h_{r}(x_{r})}{\lVert\tau_{r}\rVert_{1}}\left(\prod_{i=r+1}^{s}\frac{g_{i}(x_{i})}{\lVert \mu_{i}\rVert_{1}} \right).
	\end{align*}
\end{proof}

\subsection{Counting trivial solutions}

We close this section by exhibiting a second consequence of Theorem \ref{thmQuasiCtrl}: we can bound the number of trivial solutions to (\ref{thmCLP}). In particular, we show that the set of trivial solutions is sparse in the set of all solutions. From this it follows that one can obtain non-trivial solutions to (\ref{thmCLP}) over a set $S$ by taking $N$ sufficiently large and showing that a positive proportion of all solutions $\vx\in[N]^{s}$ lie in $(S\cap[N])^{s}$.

\begin{theorem}\label{thmTrivSol}
	Let $k,n,s\in\N$ with $k\geqslant 2$. Let $\vM=(a_{i,j})\in\Z^{n\times s}$ be a matrix of rank $n$ with no zero columns. Let $\delta=\delta(k,n)=2kn(k^{2}n+1)^{-1}$. If $\vM$ satisfies condition {\upshape(\ref{cndNonSing})} of Theorem \ref{thmMain}, then there are $O_{n,s}(N^{s-kn+\delta-1})$ trivial solutions $\vx\in[N]^{s}$ to (\ref{eqnCLPsys}).
\end{theorem}
\begin{proof}
	By the union bound, it suffices to show that $\cN_{\vM}(u,v;N)\ll_{\vM}N^{s-kn+\delta-1}$ holds for all $u,v\in[s]$ with $u<v$, provided $N\in\N$ is sufficiently large. Here, $\cN_{\vM}(u,v;N)$ denotes the number of solutions $\vx\in[N]^{s}$ to (\ref{eqnCLPsys}) with $x_{u}=x_{v}$.
	
	By Lemma \ref{lemPermStr}, we may assume that $\vM$ is a block upper triangular matrix with diagonal $(\vM_{1},\ldots,\vM_{r})$, where each $\vM_{i}\in\Z^{n_{i}\times s_{i}}$ is quasi-$k^{2}$-partitionable and satisfies $s_{i}>k^{2}n_{i}$. As in the statement of Theorem \ref{thmQuasiCtrl}, denote the columns of $\vM$ by $\vc^{(i,j)}$. Hence, we can find indices $i_{u},i_{v}\in[r]$, $j_{u}\in[s_{i_{u}}]$, and $j_{v}\in[s_{i_{v}}]$ such that $\vc^{(i_{u},j_{u})}$ and $\vc^{(i_{v},j_{v})}$ correspond to the $u$th and $v$th columns of $\vM$ respectively. By orthogonality, we observe that
	\begin{equation*}
	\cN_{\vM}(u,v;N)=\int_{\T^{n}}\fkF_{k}[1_{[N]}](\valpha\cdot(\vc^{(i_{u},j_{u})}+\vc^{(i_{v},j_{v})}))\prod_{i=1}^{r}\prod_{j=1}^{s_{i}}\psi_{i,j}(\valpha\cdot\vc^{(i,j)})d\valpha,
	\end{equation*}
	where, for all $\alpha\in\T$,
	\begin{equation*}
	\psi_{i,j}(\alpha):=
	\begin{cases}
	1,\hspace{1mm}&\text{if }(i,j)\in\{(i_{u},j_{u}),(i_{v},j_{v})\};\\
	\fkF_{k}[1_{[N]}](\alpha), &\text{otherwise}.
	\end{cases}
	\end{equation*}
	Thus, by setting $p_{i}:=s_{i}/n_{i}$, Theorem \ref{thmQuasiCtrl} provides us with the upper bound
	\begin{align*}
	\cN_{\vM}(u,v;N)&\leqslant\lVert\fkF_{k}[1_{[N]}]\rVert_{\infty}\prod_{i=1}^{r}\prod_{j=1}^{s_{i}}\lVert\psi_{i,j}\rVert_{L^{p_{i}}(\T)}.
	\end{align*}
	For each $i\in[r]$ and $j\in[s_{i}]$ with $(i,j)\notin\{(i_{u},j_{u}),(i_{v},j_{v})\}$, Lemma \ref{lemQuadRest} gives
	\begin{equation*}
	\lVert\psi_{i,j}\rVert_{L^{p_{i}}(\T)}\ll_{n,s} N^{1-\frac{k}{p_{i}}}=N^{1-\frac{kn_{i}}{s_{i}}}.
	\end{equation*}
	Hence, on noting that $\lVert\fkF_{k}[1_{[N]}]\rVert_{\infty}=N$, we find that
	\begin{equation*}
	\cN_{\vM}(u,v;N)\ll_{n,s}N^{s-kn+1}\cdot N^{-2+\frac{k}{p_{i_{u}}}+\frac{k}{p_{i_{v}}}}.
	\end{equation*}
	The result now follows on noting that $p_{i}\geqslant k^{2}+\tfrac{1}{n_{i}}\geqslant k^{2}+\tfrac{1}{n}$.
\end{proof} 

\section{Induction on Colours}\label{secIndCol}

The goal of this section is to show that the task of establishing partition regularity for the system (\ref{eqnCLPsys}) can be accomplished by counting solutions over dense sets and multiplicatively syndetic sets. We begin by recalling the definition of a multiplicatively syndetic set.

\begin{definition}[Multiplicatively syndetic sets]
	Let $M\in\N$. A set $S\subseteq\N$ is called \emph{multiplicatively $[M]$-syndetic} if $S\cap\{x,2x,\ldots,Mx\}\neq\emptyset$ for every $x\in\N$. We say that $S$ is a \emph{multiplicatively syndetic set} if $S$ is multiplicatively $[M]$-syndetic for some $M\in\N$.
\end{definition}

\begin{remark}
	In \cite{clp}, multiplicatively $[M]$-syndetic set are also called \emph{$M$-homogeneous sets}.
\end{remark}

Chow, Lindqvist, and Prendiville \cite{clp} observed that a homogeneous system of equations such as (\ref{eqnCLPsys}) is partition regular if it has a solution over every multiplicatively syndetic set. In fact, one can show that the converse statement is also true, see \cite{chapman}. This argument enables us to reduce Theorem \ref{thmMain} to the following.

\begin{theorem}\label{thmTredMain}
	Let $M\in\N$, and let $\vM=(a_{i,j})$ be a $n\times s$ integer matrix of rank $n$. Let $k\geqslant 2$. If $\vM$ satisfies the columns condition and condition {\upshape(\ref{cndNonSing})} of Theorem \ref{thmMain}, then there exist positive constants $N_{0}=N_{0}(k,M;\vM)\in\N$ and $c_{0}=c_{0}(k,M;\vM)>0$ such that the following is true. If $S\subseteq\N$ is a multiplicatively $[M]$-syndetic set, and $N\geqslant N_{0}$, then there are at least $c_{0}N^{s-kn}$ non-trivial solutions $\vx\in (S\cap [N])^{s}$ to (\ref{eqnCLPsys}).
\end{theorem}
\begin{proof}[Proof of Theorem \ref{thmMain} given Theorem \ref{thmTredMain}]
	Note that we may assume $\vM$ has rank $n$ by deleting any linearly dependent rows. The result now follows by the induction on colours argument given in \cite[\S4.2]{clp} and \cite[\S13.2]{clp}.
\end{proof}

Observe that any diagonal polynomial equation of degree $k$ which satisfies the columns condition can be written in the form
\begin{equation}\label{eqnTeqnCLP}
\sum_{i=1}^{s}a_{i}x_{i}^{k}=\sum_{j=1}^{t}b_{j}y_{j}^{k},
\end{equation}
for some integers $s\in\N$, $t\geqslant 0$, and $a_{1},\ldots,a_{s},b_{1},\ldots,b_{t}\in\Zno$ are such that $a_{1}+\cdots+a_{s}=0$.
For equations of this form, Chow, Lindqvist, and Prendiville \cite{clp} establish partition regularity (Theorem \ref{thmCLP}) by showing that one can find (many) solutions to (\ref{eqnTeqnCLP}) with $x_{i}$ lying in a dense set of smooth numbers and $y_{j}$ lying in a multiplicatively syndetic set.

We seek a similar reformulation for the systems considered in Theorem \ref{thmMain}. We start by listing a number of equivalent definitions of the columns condition.

\begin{proposition}\label{propColCond}
	Let $\vM\in\Z^{h\times k}$ be a matrix of rank $h$. Then the following are all equivalent.
	\begin{enumerate}[\upshape(i)]
		\item\label{item1} $\vM$ obeys the columns condition;
		\item\label{item2} the system of equations $\vM\vx=\vzero$ is partition regular;
		\item\label{item3} for every $\vv=(v_{1},\ldots,v_{k})\in\Q^{k}\setminus\{\vzero\}$ in the row space of $\vM$, there exists a non-empty set $J\subseteq[k]$ such that $v_{j}\neq 0$ for all $j\in J$, and $\sum_{j\in J}v_{j}=0$;
		\item\label{cndABC} there exist positive integers $n,s\in\N$ and non-negative integers $m,t\in\Z_{\geqslant 0}$ with $h=n+m$ and $k=s+t$ such that the following is true. The matrix $\vM$ is equivalent to a matrix of the form
		\begin{equation}\label{eqnMatrixForm}
		\begin{pmatrix}
		\vA & \vB\\
		\vzero & \vC
		\end{pmatrix}
		,
		\end{equation}
		for some matrices $\vA\in\Z^{n\times s}, \vB\in\Z^{n\times t}, \vC\in\Z^{m\times t}$ such that $\vA$ is a matrix of rank $n$ whose columns sum to $\vzero$, and $\vC$ obeys the columns condition.\footnote{Note that $\vC$ is allowed to contain zero columns.}
	\end{enumerate}
\end{proposition}
\begin{proof}
	The equivalence of (\ref{item1}) and (\ref{item2}) is provided by \cite[Satz IV]{rado}, whilst the equivalence of (\ref{item2}) and (\ref{item3}) follows from \cite[Lemma 4]{RadSys}.
	
	Suppose that $\vM$ satisfies (\ref{cndABC}), and let $\vM'$ be the matrix given in (\ref{eqnMatrixForm}). Let $\vv=(v_{1},\ldots,v_{k})\in\Q^{k}$ be a non-zero vector in the row space of $\vM'$. Since the columns of $\vA$ sum to $\vzero$, we see that $v_{1}+\ldots+v_{s}=0$. Thus, if $v_{1},\ldots,v_{s}$ are not all zero, then we may take $J=\{i\in[s]: v_{i}\neq 0\}\neq\emptyset$. If, on the other hand, we have $v_{1}=\ldots=v_{s}=0$, then the hypothesis that $\vA$ has full rank implies that $(v_{s+1},\ldots,v_{k})\in\Q^{t}$ is a non-zero element of the row space of $\vC$. Since property (\ref{item3}) holds for $\vC$, we deduce that there exists $J\subseteq[k]\setminus[s]$ such that $\sum_{j\in J}v_{j}=0$ and $v_{j}\neq 0$ for all $j\in J$. We therefore find that (\ref{item3}) holds for $\vM'$. Since vectors in the row space of $\vM$ are just permutations of vectors in the row space of $\vM'$, we conclude that (\ref{item3}) holds for $\vM$.
	
	Finally, suppose that (\ref{item1}) and (\ref{item3}) both hold for $\vM$. Let $I=J_{1}$ and $J=[k]\setminus I$, where $[k]=J_{1}\cup\cdots\cup J_{p}$ is the partition provided by the columns condition. By permuting the columns of $\vM$, we may assume that $I=[s]$ for some $0<s\leqslant k$. Let $\vM'$ be the $h\times s$ submatrix of $\vM$ formed from the columns of $\vM$ indexed by $I$. By performing elementary row operations, we may assume that the bottom $(h-n)$ rows of $\vM'$ are identically zero, where $n$ is the rank of $\vM'$. By performing these same operations to $\vM$, we see that $\vM$ is equivalent to a matrix of the form (\ref{eqnMatrixForm}). Our choice of $I$ ensures that $\vA$ has rank $n$ and that the columns of $\vA$ sum to $\vzero$. By considering only linear combinations of the bottom $h-n$ rows of the matrix (\ref{eqnMatrixForm}), we see that $\vC$ satisfies property (\ref{item3}). Hence, by the equivalence of (\ref{item1}) and (\ref{item3}), we deduce that $\vC$ obeys the columns condition. We have therefore shown that (\ref{cndABC}) holds.
\end{proof}

This proposition therefore shows that to prove Theorem \ref{thmMain} we need only consider systems of equations of the form
\begin{align}
\vA\vx^{\otimes k}&=\vB\vy^{\otimes k};\nonumber\\
\vzero&=\vC\vy^{\otimes k}.\label{eqnCptSys}
\end{align}
where $\vA,\vB,\vC$ are as described in (\ref{cndABC}) of Proposition \ref{propColCond}. Here we have also recalled the notation $(x_{1},\ldots,x_{s})^{\otimes k}:=(x_{1}^{k},\ldots,x_{s}^{k})$. Thus, Theorem \ref{thmTredMain} may be rewritten to the following.

\begin{theorem}\label{thmReTredMain}
	Let $k\in\N\setminus\{1\}$. Let $\vA\in\Z^{n\times s}, \vB\in\Z^{n\times t}$, and $\vC\in\Z^{m\times t}$, for some $n,s\in\N$ and $m,t\in\Z_{\geqslant 0}$. Let $\vM$ be the matrix given by (\ref{eqnMatrixForm}). Suppose that $\vA$ is a matrix of rank $n$ whose columns sum to $\vzero$, and that $\vC$ obeys the columns condition. If $\vM$ satisfies condition {\upshape(\ref{cndNonSing})} of Theorem \ref{thmMain}, then for every $M\in\N$, there exists $N_{0}=N_{0}(k,M;\vM)\in\N$ and $c_{0}=c_{0}(k,M;\vM)>0$ such that the following is true. If $S\subseteq\N$ is a multiplicatively $[M]$-syndetic set, and $N\geqslant N_{0}$, then there are at least $c_{0}N^{s+t-k(m+n)}$ non-trivial solutions $(\vx,\vy)\in (S\cap [N])^{s+t}$ to (\ref{eqnCptSys}).
\end{theorem}
\begin{proof}[Proof of Theorem \ref{thmTredMain} given Theorem \ref{thmReTredMain}]
	This is an immediate consequence of the equivalence between (\ref{item1}) and (\ref{cndABC}) in Proposition \ref{propColCond}.
\end{proof}

To prove Theorem \ref{thmReTredMain}, we follow the approach of Chow, Lindqvist, and Prendiville by seeking solutions to (\ref{eqnCptSys}) with the $x_{i}$ lying in a dense set, and the $y_{j}$ lying in a multiplicatively syndetic set. Such an approach has the advantage that it allows us to address both Theorem \ref{thmMain} and Theorem \ref{thmDens} simultaneously by proving a stronger result (Theorem \ref{thmQuadMain}). 

To elucidate this argument further, we first recall the fact that multiplicatively syndetic sets have positive (lower) density.

\begin{lemma}[Density of multiplicatively syndetic sets]\label{lemDensSyn}
	Let $M,N\in\N$. If $S\subseteq\N$ is multiplicatively $[M]$-syndetic, then
	\begin{equation*}
	|S\cap[N]|\geqslant\frac{1}{M}\left\lfloor\frac{N}{M}\right\rfloor.
	\end{equation*}
\end{lemma}

\begin{proof}
	See \cite[Lemma 4.2]{clp} or \cite[Lemma 3.1]{chapman}.
\end{proof}

Our earlier remarks therefore show that partition regularity is a special case of density regularity, namely the case where the dense set we seek solutions over is multiplicatively syndetic. Combining Proposition \ref{propColCond} with these observations allows us to generalise \cite[Theorem 12.1]{clp} to systems of equations. We also take this opportunity to impose a number of helpful properties on the matrices $\vM$ in order to simplify our arguments in \S\ref{secArl}.

\begin{theorem}[Dense-syndetic regularity]\label{thmQuadMain}
	Let $k\in\N\setminus\{1\}$. Let $n,s\in\N$, and $m,t\in\Z_{\geqslant 0}$. Let $\vA\in\Z^{n\times s}$,  $\vB\in\Z^{n\times t}$, and $\vC\in\Z^{m\times t}$. Let $\delta>0$, and $M\in\N$. Let $\vM$ be the matrix defined by (\ref{eqnMatrixForm}). Suppose that the following conditions all hold.
	\begin{enumerate}[\upshape(i)]
		\item\label{TcndFirst} the matrix $\vA$ has rank $n$ and no zero columns;
		\item\label{TcndSumCol} the columns of $\vA$ sum to $\vzero$;
		\item\label{TcndColC} if $m,t>0$, then $\vC$ obeys the columns condition;
		\item\label{TcndLast} the matrix $\vM$ satisfies condition {\upshape(\ref{cndNonSing})} of Theorem \ref{thmMain};
		\item\label{TcndDiag} the matrix $\vA$ contains a non-singular $n\times n$ diagonal submatrix;
		\item\label{TcndCoprime} for each $i\in[n]$, the entries in the $i$th row of $\vA$ are coprime;
		\item\label{TcndDiv} every entry of $\vB$ is divisible by every non-zero entry of $\vA$;
		\item\label{XcndIterate} if $m,t>0$, then, for any $M\in\N$ and multiplicatively $[M]$-syndetic set $S\subseteq\N$, there are $\gg_{M,\vC}N^{t-km}$ solutions $\vy\in(S\cap[N])^{t}$ to $\vC\vy^{\otimes k}=\vzero$, provided that $N$ is sufficiently large relative to $M$ and $\vC$.
	\end{enumerate}
	Then there exists a positive integer $N_{0}=N_{0}(\delta,k,M,\vM)$ and a positive constant $c_{0}=c_{0}(\delta,k,M,\vM)>0$ such that the following is true. Let $S\subseteq\N$ be a multiplicatively $[M]$-syndetic set. Let $N\in\N$, and $A\subseteq [N]$ be such that $|A|\geqslant\delta N$. If $N\geqslant N_{0}$, then there are at least $c_{0}N^{s+t-k(n+m)}$ solutions $(\vx,\vy)\in A^{s}\times (S\cap[N])^{t}$ to (\ref{eqnCptSys}).
\end{theorem}

\begin{remark}
	Condition (\ref{XcndIterate}) simply asserts that Theorem \ref{thmTredMain} is true if ``$\vM$'' and ``(\ref{eqnCLPsys})'' are replaced by ``$\vC$'' and ``$\vC\vy^{\otimes k}=\vzero$'' respectively.
\end{remark}

\begin{proof}[Proof of Theorem \ref{thmReTredMain} given Theorem \ref{thmQuadMain}]
	Let $\vA,\vB,\vC$, and $\vM$ be as given in the statement of Theorem \ref{thmReTredMain}. Observe that the hypotheses of Theorem \ref{thmReTredMain} guarantee that $\vA,\vC,$ and $\vM$ obey conditions (\ref{TcndFirst})-(\ref{TcndLast}) of Theorem \ref{thmQuadMain}. By performing elementary row operations, we can ensure that condition (\ref{TcndDiag}) is also satisfied. We henceforth assume that these five conditions all hold. 
	
	Note that, by taking $N$ sufficiently large, Theorem \ref{thmTrivSol} implies that if we can find $cN^{s+t-k(n+m)}$ solutions $(\vx,\vy)\in A^{s}\times (S\cap[N])^{t}$ to (\ref{eqnCptSys}), then we can obtain  $(c_{0}/2)N^{s+t-k(n+m)}$ non-trivial solutions over $A^{s}\times (S\cap[N])^{t}$. Thus, we can remove from the conclusion of Theorem \ref{thmReTredMain} the condition that the solutions we find are non-trivial.
	
	We now seek to show that, without loss of generality, we may assume that conditions (\ref{TcndCoprime}) and (\ref{TcndDiv}) of Theorem \ref{thmQuadMain} hold. Let $K\in\N$ denote the absolute value of the product of all the non-zero entries of $\vA$ (counting multiplicity). Let $\vB':=K^{k^{2}}\vB$, $\vC':=K^{k^{2}}\vC$, and
	\begin{equation*}
	\vM':=
	\begin{pmatrix}
	\vA & \vB'\\
	\vzero & \vC'
	\end{pmatrix}
	.
	\end{equation*}
	In other words, $\vM'$ is the result of multiplying the last $t$ columns of $\vM$ by $K^{k^{2}}$. Observe that $\vA,\vB',\vC'$, and $\vM'$ all obey conditions (\ref{TcndFirst})-(\ref{TcndDiag}) of Theorem \ref{thmQuadMain}. Let $M\in\N$ and let $S\subseteq\N$ be a multiplicatively $[M]$-syndetic set. Observe that the set $S':=\{x\in\N:K^{k}x\in S \}$ is also multiplicatively $[M]$-syndetic. Moreover, if $(\vx,\vy)\in[N]^{s}\times(S'\cap[N])^{t}$ satisfies both $\vA\vx=\vB'\vy$ and $\vC'\vy=\vzero$, then $(\vx,K^{k}\vy)\in[N]^{s}\times(S\cap[K^{k}N])^{t}$ is a solution to (\ref{eqnCptSys}). Hence, by rescaling, we deduce that the conclusion of Theorem \ref{thmReTredMain} holds for $\vA,\vB,\vC$ if it holds for $\vA,\vB',\vC'$. Furthermore, since every entry of $\vB'$ is a multiple of $K^{k}$, for each $i\in[n]$ we can divide the $i$th row of $\vM'$ by the greatest common divisor of the $i$th row of $\vA$. The resulting matrix $\vM''$ has integer entries (by choice of $K$), has the same solution set as $\vM'$, and obeys conditions (\ref{TcndFirst})-(\ref{TcndDiv}) of Theorem \ref{thmQuadMain}. Hence, by replacing $\vM$ with $\vM''$, we may henceforth assume without loss of generality that $\vA,\vB$, and $\vC$ satisfy conditions (\ref{TcndFirst})-(\ref{TcndDiv}).
	
	We now proceed by induction on $n+m$ to show that the conclusion of Theorem \ref{thmReTredMain} holds for $\vA,\vB,\vC$. Note that if $m=0$ or $t=0$, then condition (\ref{XcndIterate}) holds vacuously, and so the result follows from Theorem \ref{thmQuadMain} and Lemma \ref{lemDensSyn} (by taking $A=S\cap[N]$). In particular, this proves the result for $n+m=1$.
	
	Now suppose that $n+m\geqslant 2$ and $m,t\geqslant 1$. Assume the induction hypothesis that the conclusion of Theorem \ref{thmReTredMain} holds for any $\tilde{\vA}\in\Z^{n'\times s'}, \tilde{\vB}\in\Z^{n'\times t'}$, $\tilde{\vC}\in\Z^{m'\times t'}$ with $1\leqslant n'+m'<n+m$ which satisfy conditions (\ref{TcndFirst})-(\ref{TcndDiag}) of Theorem \ref{thmQuadMain}. Consider the system of equations $\vC\vy^{\otimes k}=\vzero$. Since $m,t\geqslant 1$, condition (\ref{TcndColC}) implies that $\vC$ obeys the columns condition. By considering linear combinations of the bottom $m$ rows of $\vM$, we also find that $\vC$ obeys condition (\ref{TcndLast}) of Theorem \ref{thmQuadMain}. Thus, the induction hypothesis and Proposition \ref{propColCond} imply that condition (\ref{XcndIterate}) holds. We therefore deduce from Theorem \ref{thmQuadMain} and Lemma \ref{lemDensSyn} that the conclusion of Theorem \ref{thmReTredMain} holds for $\vA,\vB,\vC$.
\end{proof} 

\begin{proof}[Proof of Theorem \ref{thmDens} given Theorem \ref{thmQuadMain}]
	By performing elementary row operations and deleting linearly dependent rows, we may assume that $\vM$ has rank $n$, each row of $\vM$ has coprime entries, and that $\vM$ contains a non-singular square diagonal matrix with the same number of rows as $\vM$. By taking $\vB$ and $\vC$ to be empty matrices and putting $\vA=\vM$, the result follows immediately from Theorem \ref{thmQuadMain}.
\end{proof}

\section{Linearisation and the $W$-trick}\label{secLinW}

The purpose of this section is to obtain a lower bound for the number of solutions to (\ref{eqnCptSys}) in terms of solutions to the `linearised' system
\begin{align}
\vA\vx&=\vB\vy^{\otimes k};\nonumber\\
\vzero&=\vC\vy^{\otimes k}.\label{eqnLinSys}
\end{align}
We accomplish this by using the linearisation procedure developed by Chow, Lindqvist, and Prendiville \cite[\S12]{clp}. For $k\geqslant 3$, we avoid using smooth numbers and instead use the linearisation procedure detailed in Lindqvist's thesis \cite[\S6.5]{lindqvist}. This version follows the same outline as \cite[\S12]{clp}, but with the smoothness parameter $\eta$ set to $1$ (as in the quadratic case) and using the restriction estimates given in Lemma \ref{lemQuadRest}.

The upshot of applying these methods is that we are able to prove that Theorem \ref{thmQuadMain} follows from the following linearised version.
\begin{theorem}[Dense-syndetic regularity for linearised systems]\label{thmLinMain}
	Let $\delta>0$, $k\in\N\setminus\{1\}$, and $M\in\N$. Let $n,s\in\N$, and $m,t\in\Z_{\geqslant 0}$. Let $\vA,\vB,\vC$, and $\vM$ be as defined in Theorem \ref{thmQuadMain}, and suppose that they satisfy conditions (\ref{TcndFirst})-(\ref{XcndIterate}). Then there exists a positive integer $N_{1}=N_{1}(\delta,k,M,\vM)\in\N$ and a positive constant $c_{1}=c_{1}(\delta,k,M,\vM)>0$ such that the following is true. Let $N\in\N$, and suppose $A\subseteq [N]$ is such that $|A|\geqslant\delta N$. Let $S\subseteq\N$ be a multiplicatively $[M]$-syndetic set. If $N\geqslant N_{1}$, then there are at least $c_{1}N^{s+\frac{t}{k}-(n+m)}$ solutions $(\vx,\vy)\in A^{s}\times (S\cap[N^{1/k}])^{t}$ to (\ref{eqnLinSys}).
\end{theorem}

For the rest of this section we fix a choice of $k\in\N\setminus\{1\}$, and matrices $\vA,\vB,\vC$, and $\vM$ satisfying conditions (\ref{TcndFirst})-(\ref{XcndIterate}) of Theorem \ref{thmQuadMain}. For each $r\in\N$ and for finitely supported functions $f_{1},\ldots,f_{s},g_{1},\ldots,g_{s}:\Z\to\C$, we define the counting operators
\begin{equation*}
\Lambda_{r}(f_{1},\ldots,f_{s};g_{1},\ldots,g_{s}):=\sum_{\vC\vy^{\otimes k}=\vzero}\sum_{\vA\vx^{\otimes r}=\vB\vy^{\otimes k}}f_{1}(x_{1})\cdots f_{s}(x_{s})g_{1}(y_{1})\cdots g_{t}(y_{t}).
\end{equation*}
For clarity, the outer sum is taken over all $\vy\in\Z^{t}$ satisfying $\vC\vy^{\otimes k}=\vzero$, whilst the inner sum is over all $\vx\in\Z^{s}$ such that $\vA\vx^{\otimes r}=\vB\vy^{\otimes k}$. In the case where $\vC$ is empty, the outer summation is omitted and $\Lambda_{r}$ is defined by the inner sum only. If $\vC$ and $\vB$ are both empty, then the inner sum is taken over $\vA\vx^{\otimes r}=\vzero$.
For brevity, we write $\Lambda_{r}(f_{1},\ldots,f_{s};g):=\Lambda_{r}(f_{1},\ldots,f_{s};g,\ldots,g)$, and similarly write $\Lambda_{r}(f;g):=\Lambda_{r}(f,\ldots,f;g)$. For sets $A,B$, we use the abbreviation $\Lambda_{r}(A;g_{1},\ldots,g_{t}):=\Lambda_{r}(1_{A};g_{1},\ldots,g_{t})$, and similarly for the quantities $\Lambda_{r}(f_{1},\ldots,f_{s};B)$ and $\Lambda_{r}(A;B)$.

In \S\ref{secQuasi}, we obtained a generalised von Neumann theorem (Theorem \ref{thmRelGenNeu}) for counting operators with weights satisfying $p$-restriction estimates. We now specialise this result to $\Lambda_{1}$.

\begin{lemma}[Generalised von Neumann for $\Lambda_{1}$]\label{lemGenVon}
	Let $p:=k^{2}+\tfrac{1}{2(n+m)}$, and let $\eta:=(2k^{2}(n+m)+2)^{-1}$. Let $N\in\N$, and suppose that $\nu,\mu:[N]\to\R_{\geqslant 0}$ each satisfy a $p$-restriction estimate with constant $K$. Let $f,g:[N]\to\C$, and let $D\subseteq[N^{1/k}]$. If $|f|\leqslant\nu$ and $|g|\leqslant\mu$, then
	\begin{align*}
	\left\lvert\Lambda_{1}\left(\lVert\nu\rVert_{1}^{-1} f;D\right)-\Lambda_{1}\left(  \lVert\mu\rVert_{1}^{-1} g;D\right) \right\rvert
	\ll_{K,\vM} N^{\frac{t}{k}-(n+m)}\left\lVert\frac{ \hat{f}}{\lVert\nu\rVert_{1}}-\frac{ \hat{g}}{\lVert\mu\rVert_{1}}\right\rVert_{\infty}^{\eta}.
	\end{align*}
\end{lemma}
\begin{proof}
	For each $\varphi:[N^{1/k}]\to\C$, let $Q_{\varphi}:\Z\to\C$ be the function given by
	\begin{equation*}
	Q_{\varphi}(x)=
	\begin{cases}
	\varphi(y),\hspace{1mm}&\text{if }x=y^{k}\text{ for some }y\in [N^{1/k}];\\
	0, &\text{otherwise}.
	\end{cases}
	\end{equation*}
	If $B\subseteq[N^{1/k}]$, then we write $Q_{B}:=Q_{1_{B}}$. Hence, if $h:\Z\to\C$ has finite support, then we may write
	\begin{equation*}
	\Lambda_{1}(h;D)=\sum_{\vM\vx=\vzero}h(x_{1})\cdots h(x_{s})Q_{D}(x_{s+1})\cdots Q_{D}(x_{s+t}).
	\end{equation*}
	Let $Q:=Q_{[N^{1/k}]}$. Note that $\lVert Q_{B}\rVert_{1}=|B|$ and $Q_{B}\leqslant Q$ for all $B\subseteq[N^{1/k}]$. Moreover, if $h:\Z\to\C$ satisfies $|h|\leqslant Q$, then $h=Q_{H}$, where $H:[N^{1/k}]\to\C$ is defined by $H(x)=h(x^{k})$. Thus, by recalling from \S\ref{subsecFourier} that $\hat{Q}_{\varphi}=\fkF_{k}[\varphi]$, 
	Lemma \ref{lemQuadRest} shows that $Q$ satisfies a $p$-restriction estimate with constant $O_{p}(1)$ for all $p>k^{2}$. Now let $\nu_{i}=\nu$ and $\mu_{i}=\mu$ for all $i\in[s]$, and let $\nu_{j}=\mu_{j}=Q$ for all $s<j\leqslant s+t$. The result now follows from Theorem \ref{thmRelGenNeu}.
\end{proof}

\subsection{The $W$-trick}

Observe that one could attempt to linearise Theorem \ref{thmQuadMain} by equating solutions $(\vx,\vy)\in A^{s}\times S^{t}$ to (\ref{eqnCptSys}) with solutions $(\vx^{\otimes k},\vy)\in B^{s}\times S^{t}$ to (\ref{eqnLinSys}), where $B=\{x^{k}:x\in A\}$. The most immediate problem with this approach is that the number of solutions sought are different; we need to find $\gg_{M,\vM}N^{s+t-k(n+m)}$ solutions $(\vx,\vy)\in (A^{s}\times S^{t})\cap[N]^{s+t}$ to (\ref{eqnCptSys}), and $\gg_{M,\vM}N^{s+\frac{t}{k}-(n+m)}$ solutions $(\vx,\vy)\in A^{s}\times (S\cap[N^{1/k}])^{t}$ to (\ref{eqnLinSys}). This issue can be resolved by taking a weighted count of solutions to (\ref{eqnLinSys}). We return to this topic in the next subsection.

The second problem that arises comes from the fact that, in general, the $k$th powers are not uniformly distributed modulo $p$ for all primes $p$. This means that the Fourier coefficients of the indicator function of the $k$th powers in $[N]$ differ significantly from the Fourier coefficients of a weighted indicator function of $[N]$. This prevents us from making use of the Generalised von Neumann theorem to compare solutions of (\ref{eqnCptSys}) with those of (\ref{eqnLinSys}). The principle used to fix this problem is known as the \emph{$W$-trick}. Originally introduced by Green \cite{greenW} to solve equations in primes, a $W$-trick for squares was subsequently developed by Browning and Prendiville \cite{brownprend} and later generalised to squares and (smooth) higher powers by Chow, Lindqvist, and Prendiville \cite[\S12]{clp}.

We now describe the steps of the $W$-trick. Given $w\in\N$, define $W\in\N$ by
\begin{equation}\label{defW}
W=W(k,w):=k^{k-1}\prod_{p\leqslant w}p^{k},
\end{equation}
where the product is taken over all primes which do not exceed $w$. To transfer from a dense subset of $[N]$ to a subprogression, we require the following technical lemma.

\begin{lemma}[{\cite[Lemma A.4]{clp}}]\label{lemZetaXi}
	Let $\delta>0$, $w\in\N$, and let $W\in\N$ be defined by (\ref{defW}). Let $N\in\N$, and let $A\subseteq[N]$ be such that $|A|\geqslant\delta N$. There exist positive integers $\xi,\zeta\in\N$ (which depend on $A$) with $\xi\in[W]$ and $\zeta\ll_{\delta, w} 1$ which satisfy the following three properties.
	\begin{itemize}
		\item $\xi$ and $W$ are coprime;
		\item there does not exist a prime $p>w$ which divides $\zeta$;
		\item 
		$\lvert\{x\in\Z:\zeta(\xi+Wx)\in A\}\rvert \geqslant\tfrac{1}{2}\delta\lvert\{x\in\Z:\zeta(\xi+Wx)\in [N]\}\rvert$.
	\end{itemize} 
\end{lemma}

Let $\delta>0$ and $N,w\in\N$. Suppose that we are given sets $A\subseteq[N]$ and $S\subseteq\N$ such that $|A|\geqslant\delta N$ and $S$ is multiplicatively $[M]$-syndetic. Let $\zeta$ and $\xi$ be the positive integers provided by the above lemma. Define sets $A_{1},S_{1}\subseteq\N$ by
\begin{align*}
A_{1}&:=\left\lbrace\frac{(Wz+\xi)^{k}-\xi^{k}}{kW}\in\N:\zeta(Wz+\xi)\in A\setminus\{\zeta\xi\} \right\rbrace;\\ 
S_{1}&:=\{y\in\N:\zeta(kW)^{1/k}y\in S \}.
\end{align*}
Observe that $A_{1}\subseteq[X]$, where $X$ is the positive rational number defined by
\begin{equation}\label{eqnXdef}
X=X(k,N,w,\zeta):=\frac{N^{k}}{kW\zeta^{k}}.
\end{equation}
Lemma \ref{lemZetaXi} implies that if $N\geqslant 2\zeta\xi$, then
\begin{equation*}
|A_{1}|\geqslant \frac{\delta}{2}\left(\frac{N}{\zeta W}-\frac{\xi}{W} \right)\geqslant \frac{\delta N}{4\zeta W}. 
\end{equation*}
Our aim is to count solutions to (\ref{eqnCptSys}) over $A^{s}\times (S\cap[N])^{t}$ by counting solutions to (\ref{eqnLinSys}) over $A_{1}^{s}\times (S_{1}\cap[X^{1/k}])^{t}$. This leads to the following result.

\begin{proposition}\label{propCrudeTransfer}
	Let $M,N,w\in\N$, and let $\delta>0$. Let $A\subseteq[N]$ be such that $|A|\geqslant\delta N$. Let $W,\xi,\zeta\in\N$ be as given in Lemma \ref{lemZetaXi}, and let $X$ be defined by (\ref{eqnXdef}). Let $S\subseteq\N$ be a multiplicatively $[M]$-syndetic set. Let $A_{1},S_{1}\subseteq\N$ be defined as above. Then
	\begin{equation*}
	\Lambda_{1}\left(A_{1};S_{1}\cap[X^{1/k}]\right) \leqslant\Lambda_{k}\left(A;S\cap[N]\right) .
	\end{equation*}
\end{proposition}
\begin{proof}
	Suppose that $\vx\in A_{1}^{s}$ and $\vy\in (S_{1}\cap[X^{1/k}])^{t}$ are solutions to (\ref{eqnLinSys}). We can therefore find $z_{i}\in\N$ for each $i\in[s]$ such that $kWx_{i}=(Wz_{i}+\xi)^{k}-\xi^{k}$. Let $u_{i}=\zeta(Wz_{i}+\xi)$ for each $i\in[s]$, and $v_{j}=\zeta(kW)^{1/k}y_{j}$ for each $j\in[t]$. Our construction of $A_{1}$ and $S_{1}$ shows that $\vu\in A^{s}$ and $\vv\in S^{t}$. Furthermore, we see from (\ref{eqnXdef}) that $\vv\in[N]^{t}$.
	
	Since $\vv$ is a scalar multiple of $\vy$, we find that $\vC\vv^{\otimes k}=\vzero$. Now let $i\in[n]$. Since $a_{i,1}+\cdots+a_{i,s}=0$, we deduce that
	\begin{align*}
	\sum_{j=1}^{s}a_{i,j}u_{j}^{k}&=\sum_{j=1}^{s}a_{i,j}\zeta^{k}\left((Wz_{j}+\xi)^{k}-\xi^{k}\right)+(\zeta\xi)^{k}\sum_{j=1}^{s}a_{i,j}\\
	&=k\zeta^{k}W\sum_{j=1}^{s}a_{i,j}x_{j}\\
	&=\sum_{j=1}^{s}b_{i,j}v_{j}^{k}.
	\end{align*}
	We therefore conclude that $\vA\vu^{\otimes k}=\vB\vv^{\otimes k}$. Since the map $(\vx,\vy)\mapsto(\vu,\vv)$ described above is injective, the desired result may now be obtained from a change of variables. 
\end{proof}

\subsection{Weighted solutions} 
We now turn our attention to the problem of counting weighted solutions to (\ref{eqnLinSys}). We begin by considering bounded weights. In this case, such sums can be handled by performing a minor modification to Theorem \ref{thmLinMain}. 

\begin{lemma}[Functional Theorem \ref{thmLinMain}]\label{lemFuncLin}
	Let $\delta>0$ and $M\in\N$. If Theorem \ref{thmLinMain} is true, then there exist constants $N_{0}(\delta,k,M,\vM)\in\N$ and $c_{0}(\delta,k,M,\vM)>0$ such that the following is true. Let $f:[N]\to[0,1]$, and let $S\subseteq\N$ be a multiplicatively $[M]$-syndetic set. If $N\geqslant N_{0}(\delta,k,M,\vM)$ and $\lVert f\rVert_{1}\geqslant\delta N$, then
	\begin{equation*}
	\Lambda_{1}\left(f;S\cap[N^{1/k}]\right)\geqslant c_{0}(\delta,k,M,\vM)N^{s+\frac{t}{k}-(n+m)}.
	\end{equation*}
\end{lemma}
\begin{proof}
	This result follows from exactly the same argument used to prove both \cite[Lemma 5.2]{clp} and \cite[Lemma 11.3]{clp}.
\end{proof}

Let $N,w\in\N$, and $\delta>0$. Let $W,\xi,\zeta\in\N$  be as given in Proposition \ref{propCrudeTransfer}, and let $X$ be given by (\ref{eqnXdef}).
The weight function $\nu:[X]\to\R_{\geqslant 0}$ is defined by
\begin{equation}\label{eqnWgtdef}
\nu(n):=
\begin{cases}
x^{k-1},\hspace{1mm}&\text{if }n=\frac{x^{k}-\xi^{k}}{kW}\hspace{1mm}\text{for some }x\in[N/\zeta]\text{ with }x\equiv\xi\bmod W;\\
0, &\text{otherwise}.
\end{cases}
\end{equation}
We now record some of the pseudorandomness properties possessed by the weight $\nu$, as given in \cite[\S6.5]{lindqvist} (see also \cite[\S6]{clp} and \cite[\S12]{clp}).

\begin{proposition}\label{propPseudo}
	Let $N,w\in\N$, and $\delta>0$. Let $W,\xi,\zeta\in\N$ and $A_{1}\subseteq\N$ be as given in Proposition \ref{propCrudeTransfer}. Let $X\in\R_{\geqslant 0}$ and $\nu:[X]\to\R_{\geqslant 0}$ be defined by (\ref{eqnXdef}) and (\ref{eqnWgtdef}) respectively. If $N$ is sufficiently large with respect to $w$ and $\delta$, then the following properties all hold.
	\begin{itemize}
		\item (Density transfer).
		\begin{equation*}
		\sum_{n\in A_{1}}\nu(n)\gg_{k} \delta^{k}\lVert \nu\rVert_{L^{1}(\Z)};
		\end{equation*}
		\item(Fourier decay).
		\begin{equation*}
		\left\lVert \hat{\nu}-\hat{1}_{[X]}\right\rVert_{\infty}\ll Xw^{-1/k};
		\end{equation*}
		\item(Restriction estimate). For any $f:\Z\to\C$ and $p\in\R$ such that $p>k^{2}$ and $|f|\leqslant\nu$, we have
		\begin{equation*}
		\lVert\hat{f}\rVert_{p}^{p}=\int_{\T}\left\lvert \hat{f}(\alpha)\right\rvert^{p}d\alpha\ll_{p}X^{p-1}.
		\end{equation*}
	\end{itemize}
\end{proposition}
\begin{proof}
	The above properties are given by the results \cite[Lemma 6.5.2]{lindqvist}, \cite[Lemma 6.5.3]{lindqvist}, and \cite[Lemma 6.5.4]{lindqvist} respectively.
\end{proof}

These properties allow us to approximate functions majorised by $\nu$ with functions majorised by $1_{[X]}$. This enables us to transfer solutions from the linearised setting (\ref{eqnLinSys}) to the squares (\ref{eqnCptSys}). To achieve this, we require the following technical lemma.

\begin{lemma}[Dense model lemma]\label{lemModDens}
	Let $N,w\in\N\setminus\{1\}$, and $\delta>0$. Let $X\in\R_{\geqslant 0}$ and $\nu:[X]\to\R_{\geqslant 0}$ be as given in Proposition \ref{propPseudo}. Then for any function $f:[X]\to\R_{\geqslant 0}$ with $f\leqslant\nu$, there exists a function $g:[X]\to\R_{\geqslant 0}$ with $\lVert g\rVert_{\infty}\leqslant 1$ such that
	\begin{equation*}
	\lVert\hat{f}-\hat{g}\rVert_{\infty}\ll X(\log w)^{-3/2}.
	\end{equation*}
\end{lemma}
\begin{proof}
	Using the Fourier decay estimate given in Proposition \ref{propPseudo}, the result follows by applying \cite[Theorem 5.1]{prendfour} to $\nu$.
\end{proof}

This lemma allows us to consider unbounded weights $f$ for our counting operators by replacing them with bounded weights $g$ and applying Lemma \ref{lemFuncLin}. Following the strategy used in the proof of \cite[Theorem 5.5]{clp}, we can now show that Theorem \ref{thmLinMain} implies Theorem \ref{thmQuadMain}.

\begin{proof}[Proof of Theorem \ref{thmQuadMain} given Theorem \ref{thmLinMain}]
	Given $\delta>0$ and $M\in\N$, we choose $w=w(\delta,k,M,\vM)\in\N$ to be sufficiently large. By fixing this choice of $w$, the assumption $N\gg_{\delta,k,M,\vM} 1$ allows us to ensure that $N$ and $X$ are sufficiently large relative to $\delta,M,\vM$ and $w$. 
	
	Proposition \ref{propCrudeTransfer} implies that
	\begin{equation}\label{eqnLamBd}
	\lVert\nu\rVert_{\infty}^{s}\Lambda_{k}(A;S\cap[N])\geqslant\Lambda_{1}(\nu 1_{A_{1}};S_{1}\cap[X^{1/k}]).
	\end{equation}
	Recall from (\ref{eqnXdef}) and (\ref{eqnWgtdef}) respectively that $X\gg_{\delta,k,M}N^{k}$ and $\lVert\nu\rVert_{\infty}\leqslant N^{k-1}$.
	We may therefore deduce Theorem \ref{thmQuadMain} from (\ref{eqnLamBd}) if we can prove that
	\begin{equation}\label{eqnToprove}
	\Lambda_{1}(\nu 1_{A_{1}};S_{1}\cap[X^{1/k}])\gg_{\delta,k,M,\vM} X^{s+\frac{t}{k}-(n+m)}.
	\end{equation} 
	
	Let $f:=\nu 1_{A_{1}}$. By choosing $N$ sufficiently large with respect to $w$ and $\delta$, the density transfer estimate in Proposition \ref{propPseudo} implies that $\lVert f\rVert_{1}\gg\delta^{k}\lVert\nu\rVert_{1}$.
	 Lemma \ref{lemModDens} provides us with a function $g:[X]\to[0,1]$ such that
	\begin{equation*}
	\lVert\hat{f}-\hat{g}\rVert_{\infty}\ll X(\log w)^{-3/2}.
	\end{equation*}
	Note that the Fourier decay estimate given in Proposition \ref{propPseudo} implies that
	\begin{equation}\label{eqnConvNu}
	\left\lvert \lVert\nu\rVert_{1}-X\right\rvert\leqslant\left\lVert \hat{\nu}-\hat{1}_{[X]}\right\rVert_{\infty}\ll Xw^{-1/k}.
	\end{equation}
	Thus, if $w$ is sufficiently large, then
	\begin{equation*}
	\left\lVert\frac{\hat{f}}{\lVert\nu\rVert_{1}}-\frac{\hat{g}}{X}\right\rVert_{\infty}\ll (\log w)^{-3/2}.
	\end{equation*}
	By considering the Fourier coefficients at $0$, if $w$ is sufficiently large, then the above inequality implies that $\lVert g\rVert_{1}\gg\delta^{k} X$. Hence, from Lemma \ref{lemFuncLin} it follows that
	\begin{equation}\label{eqnGbd}
	\Lambda_{1}(g;S\cap[X^{1/k}])\gg_{\delta,k,M,\vM}X^{s+\frac{t}{k}-(n+m)}.
	\end{equation}
Taking $p=k^{2}+\tfrac{1}{2(n+m)}$ and $D=S\cap[X^{1/k}]$, Lemma \ref{lemGenVon} gives the bound
	\begin{equation*}
	\left\lvert \Lambda_{1}(\lVert\nu\rVert_{1}^{-1}f;D)-\Lambda_{1}(X^{-1}g;D)\right\rvert\ll_{\delta,k,M,\vM}X^{\frac{t}{k}-(n+m)}(\log w)^{-\eta},
	\end{equation*}
	where $\eta=\tfrac{3}{4}(k^{2}(n+m)+1)^{-1}\in(0,1)$.
	We may therefore deduce (\ref{eqnToprove}) from (\ref{eqnConvNu}) and (\ref{eqnGbd}) upon choosing $w$ to be sufficiently large relative to the implicit constants appearing in these two inequalities and the above.
\end{proof}

It should be emphasised that the above proof shows that if the conclusion of Theorem \ref{thmLinMain} holds for a given matrix $\vM$, then the conclusion of Theorem \ref{thmQuadMain} holds for the same matrix $\vM$. In particular, if $\vB$ and $\vC$ are empty ($m=t=0$), then the above proof provides us with a means to establish density regularity for $\vA\vx^{\otimes k}=\vzero$. This method gives an alternative proof of Theorem \ref{thmDens} which does not require us to first prove Theorem \ref{thmLinMain} in full generality. All that we require is the following theorem of Frankl, Graham, and R\"{o}dl \cite{franklgrahamrodl}.

\begin{theorem}[{\cite[Theorem 2]{franklgrahamrodl}}]\label{thmFrankDens}
	Let $\vM\in\Z^{n\times s}$ be an integer matrix whose columns sum to $\vzero$. If there exists at least one non-trivial solution $\vy\in\N^{s}$ to the system $\vM\vy=\vzero$, then there exist constants $N_{1}=N_{1}(\delta,\vM)\in\N$ and $c_{1}=c_{1}(\delta,\vM)>0$ such that the following is true. If $N\geqslant N_{1}$, then for any $A\subseteq[N]$ such that $|A|\geqslant\delta N$, there are at least $c_{0}N^{s-n}$ non-trivial solutions $\vx\in A^{s}$ to the system of equations $\vM\vx=\vzero$.
\end{theorem}

\begin{proof}[Proof of Theorem \ref{thmDens}]
	As shown at the end of \S\ref{secIndCol}, Theorem \ref{thmDens} follows from the case $m=t=0$ of Theorem \ref{thmQuadMain}. By the argument above that Theorem \ref{thmLinMain} implies Theorem \ref{thmQuadMain}, it only remains to prove Theorem \ref{thmLinMain} in the case where $m=t=0$. This result follows immediately from Theorem \ref{thmFrankDens}.
\end{proof}

\section{Arithmetic regularity}\label{secArl}

The objective of this final section is to use the arithmetic regularity lemma to prove Theorem \ref{thmLinMain}. For the rest of this section, we fix a choice of matrices $\vA\in\Z^{n\times s}$, $\vB\in\Z^{n\times t}$, $\vC\in\Z^{m\times t}$, $\vM\in\Z^{(n+m)\times (s+t)}$ as given in Theorem \ref{thmQuadMain}, and assume that they satisfy conditions (\ref{TcndFirst})-(\ref{XcndIterate}). In particular, by permuting columns, we may assume that the first $n$ columns of $\vA$ form a non-singular diagonal $n\times n$ matrix.

Our general strategy is similar to the method used to prove Roth's theorem in \cite[\S1.2]{taoHOFA}. The regularity lemma enables us to decompose the indicator function of our dense set $A$ into more manageable functions. To describe these functions, we require the following definition. 

\begin{definition}[Lipschitz function]
	A function $F:\T^{d}\to[0,1]$ is called an \emph{$L$-Lipschitz function} if, for all $\valpha,\vbeta\in\T^{d}$, we have
	\begin{equation*}
	|F(\valpha)-F(\vbeta)|\leqslant L\lVert\valpha-\vbeta\rVert.
	\end{equation*}
\end{definition}

For our work, we only require the `abelian' arithmetic regularity lemma. This result, originally introduced by Green \cite{greenARL}, is a special case of the general arithmetic regularity lemma (see for instance \cite{gtARL}). The following version is a combination of \cite[Theorem 5]{eberhard} and \cite[Proposition 4.2]{greenLind} (see also \cite[Theorem 1.2.11]{taoHOFA}).

\begin{lemma}[Abelian arithmetic regularity lemma]\label{lemARL}
	Let $\cF:\R_{\geqslant 0}\to\R_{\geqslant 0}$ be a monotone increasing function, and let $\eps>0$. Then there exists a positive integer $L_{0}(\eps,\cF)\in\N$ such that the following is true. If $f:[N]\to[0,1]$ for some $N\in\N$, then there is a positive integer $L\leqslant L_{0}(\eps,\cF)$ and a decomposition
	\begin{equation*}
	f=f_{\str}+f_{\sml}+f_{\unf}
	\end{equation*}
	of $f$ into functions $f_{\str},f_{\sml},f_{\unf}:[N]\to[-1,1]$ such that:
	\begin{enumerate}[\upshape(I)]
		\item\label{itemNon} the functions $f_{\str}$ and $f_{\str}+f_{\sml}$ take values in $[0,1]$;
		\item the function $f_{\sml}$ obeys the bound
		\begin{equation*}
		\lVert f_{\sml}\rVert_{L^{2}(\Z)}\leqslant\eps\lVert 1_{[N]}\rVert_{L^{2}(\Z)}\, ;
		\end{equation*}
		\item the function $f_{\unf}$ obeys the bound
		\begin{equation*}
		\lVert \hat{f}_{\unf}\rVert_{\infty}\leqslant\lVert \hat{1}_{[N]}\rVert_{\infty}/\cF(L)\, ;
		\end{equation*}
		\item\label{itemStr} there exists a positive integer $d\leqslant L$, a phase $\vtheta\in\T^{d}$, and an $L$-Lipschitz function $F:\T^{d}\to[0,1]$ such that $F(x\vtheta )=f_{\str}(x)$ for all $x\in[N]$. 
	\end{enumerate}
\end{lemma}

We now employ the arithmetic regularity lemma to count solutions to (\ref{eqnLinSys}). Recall from \S\ref{secLinW} the counting operator
\begin{equation*}
\Lambda_{1}(f_{1},\ldots,f_{s};g_{1},\ldots,g_{t}):=\sum_{\vC\vy^{\otimes k}=\vzero}\sum_{\vA\vx=\vB\vy^{\otimes k}}f_{1}(x_{1})\cdots f_{s}(x_{s})g_{1}(y_{1})\cdots g_{t}(y_{t}).
\end{equation*}
We are interested in the quantity $\Lambda_{1}(A;S\cap[N^{1/k}])$, where $A\subseteq[N]$ satisfies $|A|\geqslant\delta N$, and $S\subseteq\N$ is multiplicatively $[M]$-syndetic. Applying the arithmetic regularity lemma with $f=1_{A}$ (for a choice of parameters $\eps,\cF$ to be specified later), we obtain a decomposition $1_{A}=f_{\str}+f_{\sml}+f_{\unf}$. 

The first step towards the proof of Theorem \ref{thmLinMain} involves removing the uniform part $f_{\unf}$ via the Generalised von Neumann theorem (Theorem \ref{thmRelGenNeu}).

\begin{lemma}[Removing $f_{\unf}$]\label{lemUnfRem}
	Let $N\in\N$, and let $A\subseteq[N]$. Let $\eps>0$, and let $\cF:\R_{\geqslant 0}\to\R_{\geqslant 0}$ be a monotone increasing function. Let $f_{\str},f_{\sml},f_{\unf}$ be the functions provided by applying Lemma \ref{lemARL} to $f=1_{A}$. Then for any $D\subseteq[N^{1/k}]$, we have
	\begin{equation*}
	\lvert\Lambda_{1}(A;D)-\Lambda_{1}(f_{\str}+f_{\sml};D)\rvert\ll_{\vM}N^{s+\tfrac{t}{k}-(n+m)}\cF(L)^{-(2k^{2}(n+m)+2)^{-1}}.
	\end{equation*}
\end{lemma}
\begin{proof}
	Let $\nu=\mu=1_{[N]}$, and note that $\lVert 1_{[N]}\rVert_{1}=N$. Furthermore, property (\ref{itemNon}) of Lemma \ref{lemARL} implies that $0\leqslant (f_{\str}+f_{\sml}),1_{A}\leqslant 1_{[N]}$. Thus, the result follows from Lemma \ref{lemGenVon}.
\end{proof}

\subsection{Auxiliary counting operators}\label{subsecAux}

Lemma \ref{lemUnfRem} shows that the main contribution to $\Lambda_{1}(A;D)$ comes from $\Lambda_{1}(f_{\str}+f_{\sml};D)$. We therefore focus our attention on finding lower bounds for this latter quantity. Rather than count all solutions to (\ref{eqnLinSys}), it is convenient for us to restrict our attention to an explicit dense subcollection of solutions. We then show that the counting operators associated with these subcollections obey a generalised von Neumann theorem with respect to the $L^{2}$ norm. This allows us to remove the $f_{\sml}$ function from our analysis.

To make this precise, let $q:=s-n-1$, and let $\{\vu^{(0)},\ldots,\vu^{(q)} \}\subseteq\Z^{s}$ denote a $\Q$-basis for $\ker(\vA)$. Note that condition (\ref{TcndSumCol}) of Theorem \ref{thmQuadMain} implies that $q\geqslant 0$. Let $u_{i,j}$ denote the $j$th entry of $\vu^{(i)}$. Since the columns of $\vA$ sum to zero, we may take $\vu^{(0)}:=(1,1,\ldots,1)$ and assume that $u_{i,1}=0$ for all $i\in[q]$. Such a basis is not uniquely defined, however we can insist that $|u_{i,j}|\ll_{\vM} 1$ for all $i$ and $j$. 

Observe that the each row of $\vB\vy^{\otimes k}$ defines an integer homogeneous diagonal polynomial of degree $k$ in the variables $y_{1},\ldots,y_{t}$. Thus, by condition (\ref{TcndDiv}), we may define for each $i\in[n]$ an integer polynomial $P_{i}\in\Z[y_{1},\ldots,y_{t}]$ by dividing the $i$th row of $\vB\vy^{\otimes k}$ by the $i$th entry of the $i$th row of $\vA$. We also set $P_{j}=0$ for all $n<j\leqslant s$. Now note that every solution to $\vA\vx=\vB\vy^{\otimes k}$ takes the form $\vx=(z_{1}+P_{1}(\vy),\ldots,z_{s}+P_{s}(\vy))$, where $(z_{1},\ldots,z_{s})\in\ker(\vA)$.

We are now ready to define our auxiliary counting operators. Let $\vy\in\Z^{t}$, and let $B\subseteq\Z$ be a finite set. For each $\vd\in B^{q}$ and $j\in[s]$, define
\begin{equation}\label{eqnQdef}
Q_{j}(\vd,\vy):=u_{1,j}d_{1}+\cdots+u_{q,j}d_{q}+P_{j}(\vy).
\end{equation}
Given functions $f_{1},\ldots,f_{s}:\Z\to\C$ with finite support, we define
\begin{equation}\label{eqnAuxOp}
\Psi_{B,\vy}(f_{1},\ldots,f_{s}):=\sum_{x\in\Z}\sum_{\vd\in B^{q}}\prod_{j=1}^{s}f_{j}(x+Q_{j}(\vd,\vy)).
\end{equation}
For brevity, we write $\Psi_{B,\vy}(f):=\Psi_{B,\vy}(f,\ldots,f)$.
Hence, since $\{\vu^{(0)},\ldots,\vu^{(q)}\}$ is a $\Q$-basis for $\ker(\vA)$, if $f_{1},\ldots,f_{s}:[N]\to[0,1]$, then
\begin{equation}\label{eqnLamPsiBd}
\Lambda_{1}(f_{1},\ldots,f_{s};S)\geqslant\sum_{\vC\vy^{\otimes k}=\vzero}\Psi_{B,\vy}(f_{1},\ldots,f_{s})1_{S}(y_{1})\cdots 1_{S}(y_{t}).
\end{equation}

We may therefore use $\Psi$ to obtain a lower bound for $\Lambda_{1}(f_{\str}+f_{\sml};D)$. In fact, by exploiting the fact that  $\lVert f_{\sml}\rVert_{2}$ is `small', we show that it is sufficient to obtain a lower bound for $\Lambda_{1}(f_{\str};D)$. To proceed in this way, we utilise
the following $L^{2}$ generalised von Neumann theorem for $\Psi$.

\begin{lemma}[Generalised von Neumann for $\Psi$]\label{lemPsiGenVon}
	Let $N\in\N$, $\vy\in\Z^{t}$, and let $B\subseteq\Z$ be a finite set. Let $\Psi$ be defined by (\ref{eqnAuxOp}). If $f,g:[N]\to[0,1]$, then
	\begin{equation*}
	\lvert\Psi_{B,\vy}(f)-\Psi_{B,\vy}(g)\rvert\leqslant s|B|^{s-n-1}N^{1/2}\lVert f-g\rVert_{2}.
	\end{equation*}
\end{lemma}
\begin{proof}
	Let $q:=s-n-1$. We show that
	\begin{equation}\label{eqnPsiCtrl}
	\lvert\Psi_{B,\vy}(f_{1},\ldots,f_{s})\rvert\leqslant N^{1/2}|B|^{q}\lVert f_{i}\rVert_{2}
	\end{equation}
	holds for all $i\in[s]$ and all $f_{1},\ldots,f_{s}:[N]\to[-1,1]$. The lemma then follows by applying this bound to the telescoping identity
	\begin{equation*}
	\Psi_{B,\vy}(f)-\Psi_{B,\vy}(g)=\sum_{r=1}^{s}\Psi_{B,\vy}(g_{1},\ldots,g_{r-1},f_{r}-g_{r},f_{r+1},\ldots,f_{s}),
	\end{equation*}
	where $f_{j}=f$ and $g_{j}=g$ for all $j\in[s]$.
	
	Let $i\in[s]$. Since the function $f_{i}$ is supported on $[N]$, the change of variables $z=x+Q_{i}(\vd,\vy)$ yields
	\begin{equation*}
	\Psi_{B,\vy}(f_{1},\ldots,f_{s})=\sum_{z\in[N]}\sum_{\vd\in B^{q}}\prod_{j=1}^{s}f_{j}(z+Q_{j}(\vd,\vy)-Q_{i}(\vd,\vy)).
	\end{equation*}
	Applying the Cauchy-Schwarz inequality with respect to $z$ gives
	\begin{equation*}
	\lvert\Psi_{B,\vy}(f_{1},\ldots,f_{s})\rvert^{2}\leqslant\lVert f_{i}\rVert_{2}^{2}\sum_{z\in[N]}\left\lvert\sum_{\vd\in B^{q}}\prod_{\substack{j=1\\ j\neq i}}^{n}f_{j}(z+Q_{j}(\vd,\vy)-Q_{i}(\vd,\vy))\right\rvert^{2}.
	\end{equation*}
	We may therefore obtain (\ref{eqnPsiCtrl}) from the bound $\lVert f_{j}\rVert_{\infty}\leqslant 1$.
\end{proof}

\begin{lemma}[Removing $f_{\sml}$]\label{lemSmlRem}
	Let the assumptions and definitions be as in Lemma \ref{lemUnfRem}. Let $\vy\in\Z^{t}$, and let $B\subseteq\Z$ be a finite set. Then we have
	\begin{equation*}
	\Psi_{B,\vy}(f_{\str}+f_{\sml})=\Psi_{B,\vy}(f_{\str})-O_{\vM}(|B|^{s-n-1}N\eps).
	\end{equation*}
\end{lemma}
\begin{proof}
	Recall from Lemma \ref{lemARL} that $f_{\str}$ and $f_{\str}+f_{\sml}$ take values in $[0,1]$. Thus, the result follows immediately from Lemma \ref{lemPsiGenVon}.
\end{proof}

\subsection{Bohr sets}

We wish to better understand the behaviour of the structured function $f_{\str}$ appearing in Lemma \ref{lemARL}. From the definition of Lipschitz functions, we see that $f_{\str}(x)\approx f_{\str}(x+y)$ holds whenever $\lVert y\vtheta\rVert$ is `small'. Such $y$ are sometimes referred to as \emph{almost periods} for $f_{\str}$, see \cite[Lemma 1.2.13]{taoHOFA}. This leads us to consider the properties of sets of such $y$, which are known as  \emph{(polynomial) Bohr sets}.

\begin{definition}[Polynomial Bohr sets]
	Let $d,h\in\N$, $\rho>0$, and let $\valpha\in\T^{d}$. The (\emph{polynomial}) \emph{Bohr set} $\rB_{h}(\valpha,\rho)$ is the set
	\begin{equation*}
	\rB_{h}(\valpha,\rho):=\bigcap_{i=1}^{d}\{n\in\N: \lVert n^{h}\alpha_{i}\rVert<\rho\}.
	\end{equation*}
\end{definition}

A key property of Bohr sets is that they have positive density. Furthermore, the density of $\rB_{h}(\valpha,\rho)$ on $[N]$ (for $N$ suitably large) can be bounded from below by a positive quantity which depends only on $d,h$ and $\rho$. A crucial aspect of this result is that this uniform lower bound does not depend on $\valpha$. 

Using Lemma \ref{lemDensSyn}, we can deduce such a result by first showing that Bohr sets are multiplicatively sydnetic. Furthermore, we prove that the intersection of a multiplicatively syndetic set with any finite intersection of non-empty Bohr sets is multiplicatively syndetic. This fact may be of independent interest. 

The key tool needed to establish these facts is the following polynomial recurrence result recorded in \cite{schmidt}.

\begin{lemma}[Polynomial recurrence]\label{lemPolyRec}
	If $h\in\N$, then there exists a constant $C=C(h)>0$ such that the following holds. If $\alpha\in\R$ and $N\in\N$, then
	\begin{equation*}
	\min_{1\leqslant n\leqslant N}\lVert n^{h}\alpha\rVert\leqslant CN^{-2^{-h}}.
	\end{equation*}
\end{lemma}

\begin{proof}
	See \cite[Theorem 7A]{schmidt}.
\end{proof}

\begin{lemma}[Syndeticity of Bohr sets]\label{lemBohrSyn}
	Let $h\in\N$ and $0<\rho\leqslant 1$. Then there exists a constant $M_{0}=M_{0}(\rho,h)\in\N$ such that the following is true. If $\alpha\in\T$, then the Bohr set $\rB_{h}(\alpha,\rho)$ is multiplicatively $M_{0}$-syndetic.
\end{lemma}
\begin{proof}
	Lemma \ref{lemPolyRec} immediately implies that there exists some $M\ll_{h}\rho^{-2^{h}}$ such that $[M]$ intersects $\rB_{h}(\beta,\rho)$ for each $\beta\in\R$. Hence, by replacing $\beta$ with $m^{h}\alpha$, we deduce that $\rB_{h}(\alpha,\rho)$ intersects $m\cdot[M]$ for every $m\in\N$. We therefore find that $\rB_{h}(\alpha,\rho)$ is multiplicatively $[M]$-syndetic.
\end{proof}
\begin{corollary}[Syndeticity for Bohr-syndetic intersections]\label{corSynBohr}
	Let $d,h,M\in\N$, and $0<\rho\leqslant 1$. There exists a constant $M_{1}=M_{1}(\rho,d,h,M)\in\N$ such that the following is true. Let $S\subseteq\N$ be a multiplicatively $[M]$-syndetic set. If $\valpha\in\T^{d}$, then the set $S\cap \rB_{h}(\valpha,\rho)$ is multiplicatively $[M_{1}]$-syndetic. Moreover, we may assume that, when considered as a function of $\rho,d,M$, the quantity $M_{1}$ satisfies
	\begin{equation}\label{eqnMDensBd}
	M_{1}(\rho,d,h,M)=\max\{M_{1}(\rho',d',h,M'):\rho'\in[\rho,1], d'\in [d], M'\in [M]\}.
	\end{equation}
\end{corollary}
\begin{proof}
	Observe that every multiplicatively $[M]$-syndetic set is multiplicatively $[M+1]$-syndetic, and $\rB_{h}((\alpha_{1},\ldots,\alpha_{d'}),\rho)\subseteq \rB_{h}((\alpha_{1},\ldots,\alpha_{d}),\rho')$ for all $\rho'\geqslant \rho$ and $d'\in[d]$. These observations show that we can guarantee (\ref{eqnMDensBd}) holds once the rest of the result is proven.
	
	By writing $S\cap \rB_{h}(\valpha,\rho)=\rB_{h}(\alpha_{1},\rho)\cap\left(S\cap \rB_{h}((\alpha_{2},\ldots,\alpha_{d}),\rho) \right)$,
	we see that the result follows by induction from the case $d=1$. Thus, it is sufficient to show that there exists some $M_{1}=M_{1}(\rho,h,M)\in\N$ such that $S\cap \rB_{h}(\alpha,\rho)$ is multiplicatively $[M_{1}]$-syndetic for all $\alpha\in\T$. 
	
	Let $\alpha\in\T$, $a\in\N$, and let $\rho'=\rho/M$. Lemma (\ref{lemBohrSyn}) provides us with some $M'=M'(\rho,h,M)\in\N$ such that $\rB_{h}(\alpha,\rho')$ is multiplicatively $[M']$-syndetic. We can therefore choose some $m\in[M']$ such that $am\in \rB_{h}(\alpha,\rho')$. Moreover, we have $am\cdot[M]\subseteq \rB_{h}(\alpha,\rho)$. Thus, $am\cdot[M]$ intersects $S\cap \rB_{h}(\alpha,\rho)$. We therefore conclude that $S\cap \rB_{h}(\alpha,\rho)$ is multiplicatively $[M\cdot M']$-syndetic.
\end{proof}

Let $\vtheta\in\T^{d}$ be as given in Lemma \ref{lemARL}, and let $0<\rho<1$. Observe that if $x,r\in[N]$ with $r\in \rB_{1}(\vtheta,\rho)$ and $x+r\in[N]$, then
\begin{equation*}
|f_{\str}(x+r)-f_{\str}(x)|=|F(\vtheta (x+r))-F(\vtheta x)|\leqslant dL\rho\leqslant L^{2}\rho.
\end{equation*}
Similarly, if $\vy\in \rB_{k}(\vtheta,\rho)^{t}$ and $j\in[s]$ are such that $x,x+P_{j}(\vy)\in[N]$, then
\begin{equation*}
|f_{\str}(x+P_{j}(\vy))-f_{\str}(x)|\ll_{\vM}L^{2}\rho.
\end{equation*}

It is important to note that we must assume $x,x+r,x+P_{j}(\vy)\in[N]$ for the above inequalities to hold. This is because we are using the convention that $f_{\str}(x)=0$ for all $x\notin[N]$. To circumvent this issue, we use a trick of Tao \cite[Lemma 1.2.13]{taoHOFA}. Rather then extend $f_{\str}$ outside of $[N]$ by $0$, we instead use the function $F$ and replace $f_{\str}(x)$ by $F(\vtheta x)$. Since $f_{\str}(x)=F(\theta x)$ only holds for $x\in[N]$, we restrict our choice of $\vy$ and $\vd$ so that $|Q_{j}(\vd,\vy)|\leqslant\rho N$ holds for all $j\in[s]$. 
Proceeding in this way produces the following result.

\begin{lemma}[Lower bound for $\Psi_{B,\vy}(f_{\str})$]\label{lemLbPsi}
	Let the assumptions and definitions be as in Lemma \ref{lemUnfRem}. Let $q:=s-n-1$, and let $\cU=\{\vu^{(0)},\ldots,\vu^{(q)} \}$ denote a $\Q$-basis for $\ker(\vA)$ with the properties described in \S\ref{subsecAux}. There exists positive constants $\rho_{0}=\rho_{0}(k,\vM,\cU)\in(0,1)$ and $N_{0}=N_{0}(k,\vM,\cU)\in\N$ such that the following is true. Let $\rho\in(0,1)$ and $N\in\N$. Let $B=\rB_{1}(\vtheta,\rho)\cap[\rho N]$, and let $\vy\in (\rB_{k}(\vtheta,\rho)\cap[(\rho N)^{1/k}])^{q}$. Let $Q_{1},\ldots,Q_{s}$ be given by (\ref{eqnQdef}), and let $\Psi_{B,\vy}$ be as defined in (\ref{eqnAuxOp}). If $N\geqslant N_{0}$ and $0<\rho\leqslant\rho_{0}$, then
	\begin{equation*}
	\Psi_{B,\vy}(f_{\str})\geqslant |B|^{q}N\left[(\delta-\eps^{2}-\cF(L)^{-1})^{s}-O_{\vM,\cU}(L^{2s}\rho^{s}+L^{2}\rho) \right].
	\end{equation*}
\end{lemma}
\begin{proof}
	Observe that our choice of $\vy$ and $B$ implies that $|Q_{i}(\vd,\vy)|\ll_{\vM,\cU}\rho N$ holds for all $i\in[s]$ and $\vd\in B^{q}$. Hence, provided $N$ and $\rho^{-1}$ are sufficiently large in terms of $k,\vM$, and $\cU$, there exists some $\rho'>0$ with $\rho'\ll_{k,\vM,\cU} \rho$ such that the following is true. Let $\Omega$ denote the set of $x\in\N$ such that $\rho' N< x< (1-\rho')N$. If $x\in\Omega$, then
	\begin{equation*}
	|f_{\str}(x+Q_{i}(\vy))-f_{\str}(x)|=|F((x+Q_{i}(\vy))\vtheta)-F(x \vtheta)|\ll_{\vM,\cU}L^{2}\rho
	\end{equation*}
	holds for all $i\in[s]$. Since $|[N]\setminus\Omega|\ll_{k,\vM,\cU} \rho N$, we can take $\rho$ sufficiently small to ensure that $\Omega\neq\emptyset$. Thus, the inequality $0\leqslant f_{\str}\leqslant 1_{[N]}$ implies that
	\begin{align}
	 \Psi_{B,\vy}(f_{\str})&\geqslant\sum_{x\in\Omega}\sum_{\vd\in B^{q}}\prod_{j=1}^{s}f_{\str}(x+Q_{j}(\vd,\vy))\nonumber\\
	 &\geqslant \sum_{\vd\in B^{q}}\left(\sum_{x=1}^{n}\left(f_{\str}(x)^{s}-O_{k,\vM,\cU}(L^{2s}\rho^{s}+L^{2}\rho)\right)-|[N]\setminus\Omega|\right)\nonumber\\
	 &\geqslant|B|^{q}N\left(N^{-1}\lVert (f_{\str})^{s}\rVert_{1}-O_{k,\vM,\cU}(L^{2s}\rho^{s}+L^{2}\rho) \right).\label{eqnFinalPsi}
	\end{align}
	By an application of H\"{o}lder's inequality, we find that
	\begin{equation*}
	\lVert (f_{\str})^{s}\rVert_{1}\geqslant N^{1-s}\left(\sum_{x=1}^{N}(1_{A}(x)-f_{\sml}(x)-f_{\unf}(x)) \right)^{s}\geqslant N(\delta-\eps^{2}-\cF(L)^{-1})^{s}.
	\end{equation*}
	Substituting the above into (\ref{eqnFinalPsi}) completes the proof.
\end{proof}

This final lemma, in combination with the previous results of this section, finally provides us with a means to prove Theorem \ref{thmLinMain}.

\begin{proof}[Proof of Theorem \ref{thmLinMain}]
	Condition (\ref{XcndIterate}) implies that if $\vC$ is non-empty, then there exist functions $\kappa:\N\to(0,1]$ and $\cM:\N\to\N$ (which depend on $k$ and $\vM$) such that the following is true. Let $M\in\N$, and let $S\subseteq\N$ be a multiplicatively $[M]$-syndetic set. If $N\geqslant\cM(M)$, then there are at least $\kappa(M)N^{t-km}$ solutions $\vy\in (S\cap[N])^{t}$ to the system $\vC\vy^{\otimes k}=\vzero$. Moreover, this condition shows that we may assume that $\cM$ is a monotone increasing function satisfying $\cM(M)\geqslant M$ for all $M\in\N$, and that $\kappa$ is monotone decreasing. In the case where $\vC$ is empty, we take $\kappa=1$ and $\cM(M)=M$ for all $M\in\N$.
	
	Let $\delta\in(0,1)$ and $M\in\N$ be fixed. Let $c=c(k,\vM)>1$ be a sufficiently small positive constant depending only on $k$ and $\vM$. Let $M_{1}=M_{1}(\rho,d,h,M)$ be the quantity given by Corollary \ref{corSynBohr}, which satisfies (\ref{eqnMDensBd}). Let $\eps:=c\delta^{s}$, and let $\cF:\R_{\geqslant 0}\to\R_{\geqslant 0}$ be a sufficiently fast growing monotone increasing function (specifically, $\cF$ is chosen so that (\ref{eqnCFdef}) holds).
	For this choice of $\cF$ and $\eps$, let $L_{0}=L_{0}(\eps,\cF)$ be the positive integer given by Lemma \ref{lemARL}. 

	Let $N\in\N$ be sufficiently large in terms of all the previously defined parameters (specifically, $N$ is chosen to satisfy (\ref{eqnFinalNbd})). Let $A\subseteq[N]$ with $|A|\geqslant\delta N$. Let $f_{\str},f_{\sml},f_{\unf}$ be the functions obtained by applying Lemma \ref{lemARL} with respect to $f=1_{A}$. Let $d,L,\vtheta$ and $F$ be as given in property (\ref{itemStr}) of Lemma \ref{lemARL}. 
	Let  $\cU=\{\vu^{(0)},\ldots,\vu^{(s-n-1)} \}$ denote a $\Q$-basis for $\ker(\vA)$ with the properties described in \S\ref{subsecAux}. Note that, by explicit computation, we can choose such a basis $\cU$ such that $|u_{i,j}|\ll_{\vM} 1$ for all $i$ and $j$. Consequently, any emergent quantities which may depend on $\cU$ are instead considered to depend on $\vM$. 
	
	Let $\rho:=c\delta^{s}L^{-2}$, and let $B:=\rB_{1}(\vtheta,\rho)\cap [\rho N]$. By choosing $c$ sufficiently small, an application of Lemma \ref{lemLbPsi} followed by Lemma \ref{lemSmlRem} reveals that
	\begin{equation*}
	\Psi_{B,\vy}(f_{\str}+f_{\sml})\geqslant\tfrac{1}{2}\delta^{s}|B|^{s-n-1}N
	\end{equation*}
	holds for all $\vy\in (\rB_{k}(\vtheta,\rho)\cap[(\rho N)^{1/k}])^{t}$. 
	
	Let $S\subseteq\N$ be a multiplicatively $[M]$-syndetic set, and let
	\begin{equation*}
	\Omega:= \left(\rB_{k}(\vtheta,\rho)\cap S\cap [(\rho N)^{1/k}]\right)^{t}\cap\{\vy\in\N^{t}:\vC\vy^{\otimes k}=\vzero \}.
	\end{equation*}
	By the non-negativity of $f_{\str}+f_{\sml}$, we deduce from (\ref{eqnLamPsiBd}) the bound
	\begin{equation}\label{eqnLamStrSml}
	\Lambda_{1}(f_{\str}+f_{\sml};S\cap[N^{1/k}])\geqslant \tfrac{1}{2}\delta^{s}|B|^{s-n-1}N|\Omega|.
	\end{equation}
	
	By Corollary \ref{corSynBohr}, the set $\rB_{1}(\vtheta,\rho)$ is multiplicatively $[M_{1}(\rho,d,1,1)]$-syndetic, whilst $\rB_{k}(\vtheta,\rho)\cap S$ is multiplicatively $[M_{1}(\rho,d,k,M)]$-syndetic. Note that Lemma \ref{lemDensSyn} implies that if $N\geqslant 2\tilde{M}^{2}$, then any every multiplicatively $[\tilde{M}]$-syndetic set has density at least $\tfrac{1}{2}\tilde{M}^{-2}$ on $[N]$. Thus, if $N$ satisfies
	\begin{equation}\label{eqnFinalNbd}
	cN^{1/k}\geqslant  \delta^{-s}L^{2}\cM\left(M_{1}(c\delta^{s}L_{0}^{-2},L_{0},1,1)+M_{1}(c\delta^{s}L_{0}^{-2},L_{0},2,M)\}\right),
	\end{equation}
	for $c$ sufficiently small, then we obtain the bounds
	\begin{align}
	|B|&\geqslant \tfrac{1}{4}c\delta^{s}L^{-2}\label{eqnLowerB} M_{1}(c\delta^{s}L^{-2},L,1,1)^{-2}N;\\
	|\Omega|&\geqslant \kappa\left(M_{1}(c\delta^{s}L^{-2},L,k,M)\right)\cdot\left(\tfrac{1}{4}c\delta^{s}L^{-2}N\right)^{\tfrac{t}{k}-m}.\label{eqnLowerOm}
	\end{align}
	
	Lemma \ref{lemUnfRem} therefore gives
	\begin{equation}\label{eqnFinalBd}
	\Lambda_{1}(A;S\cap[N^{1/k}])\geqslant \left(\gamma(\delta,L,M,\vM)-O_{k,\vM}(\cF(L)^{-\eta}) \right)N^{s+\tfrac{t}{k}-(n+m)}, 
	\end{equation}
	where $\eta:=(2k^{2}(n+m)+2)^{-1}$, and $\gamma(\delta,k,L,M,\vM)$ is the function of $\delta,k,L,M$, and $\vM$ obtained by substituting the lower bounds (\ref{eqnLowerB}) and (\ref{eqnLowerOm}) into (\ref{eqnLamStrSml}) (ignoring the factors of $N$). Using (\ref{eqnMDensBd}), we may assume that $\gamma(\delta,k,L,M,\vM)$ is a decreasing function of $L$, for fixed $\delta,k,M,\vM$. We can therefore construct a monotone increasing function $\cF_{0}:\N\to\R_{\geqslant 0}$ such that
	\begin{equation}\label{eqnCFdef}
	2C\cF_{0}(x)^{-\eta}\leqslant \gamma(\delta,k,x,M,\vM)
	\end{equation}
	holds for all $x\in\N$, where $C=C(k,\vM)>1$ is the implicit positive constant appearing in (\ref{eqnFinalBd}) (which can be assumed to be greater than $1$). We can then extend $\cF_{0}$ to a monotone increasing function $\cF:\R_{\geqslant 0}\to\R_{\geqslant 0}$ by interpolation. With this choice of $\cF$, we deduce from (\ref{eqnFinalBd}) that
	\begin{equation*}
	\Lambda_{1}(A;S\cap[N^{1/k}])\gg_{\delta,k,M,\vM}N^{s+\tfrac{t}{k}-(n+m)},
	\end{equation*}
	as required.
\end{proof}

\end{document}